\documentclass [a4paper,12pt]{article}
\usepackage{amsmath}
\usepackage{amsfonts}
\usepackage{indentfirst}
\usepackage{esint}
\usepackage{latexsym}
\usepackage{bbm}
\usepackage{amsfonts}
\usepackage{mathrsfs}
\usepackage{amssymb}
\usepackage{amsmath}
\usepackage{amsthm}
\usepackage{latexsym}
\usepackage{indentfirst}
\usepackage{enumitem}
\usepackage{bm}
\usepackage{esint}
\usepackage{cite}
\numberwithin{equation}{section}
\usepackage{color}
\allowdisplaybreaks
\usepackage{amsthm}
\usepackage{amssymb}
\usepackage{enumerate}
\usepackage{ulem}
\allowdisplaybreaks

\newtheorem{theorem}{Theorem}[section]
\newtheorem{lemma}[theorem]{Lemma}
\newtheorem{thm}[theorem]{Theorem}

\newtheorem{defn}[theorem]{Definition}
\newtheorem{rmk}[theorem]{Remark}

\makeatletter
\usepackage{amsmath}
\usepackage{amsfonts}
\usepackage{indentfirst}
\usepackage{esint}
\usepackage{latexsym}
\usepackage{color}

\usepackage{amsthm}
\usepackage{amssymb}
\usepackage{enumerate}
\usepackage{ulem}
\makeatletter

\newcommand{\Rmnum}[1]{\expandafter\@slowromancap\romannumeral #1@}
\makeatother
\begin{document}

\title{Error Estimates of Reiterated Stokes Systems via Fourier Transform Methods}
\author{Yiping Zhang\footnote{Email:zhangyiping161@mails.ucas.ac.cn}\\Academy of Mathematics and Systems Science, CAS;\\
University of Chinese Academy of Sciences;\\
Beijing 100190, P.R. China.}
\date{}
\maketitle
\begin{abstract}
In this paper, we are interested in the error estimates of the reiterated Stokes systems in a bounded $C^{1,1}$ domain with Dirichlet boundary conditions. And we have obtained the $O(\varepsilon)$ error estimates for the velocity term and $O(\varepsilon^{1/2})$ error estimates for the pressure term. Compared to the general homogenization of Stokes systems problems, the difficulty in the reiterated homogenization is that we need to handle the different scales of $x$. To overcome this difficulty, we use the Fourier transform methods which was firstly introduced by the author in \cite{Zhang1909} to separate these different scales. We also note that this method may be adapted to a more general multi-scale homogenization problem.
\end{abstract}

\section{Introduction and main results}
Before we state the introduction and the main results, we introduce the Einstein summation convention first.
Throughout this paper, we use the Einstein summation convention: an index occurring twice in a product is to be summed from 1 up to the space dimension, which means, for example,
\begin{equation*}
u_i v_i=\sum_{i=1}^{n}u_i v_i,
\end{equation*} if the space dimension is $n$.

The aim of the present paper is to study the error estimates of reiterated Dirichlet problems for Stokes systems with rapidly oscillating periodic coefficients. More precisely, let $\Omega\subset\mathbb{R}^n$ be a bounded domain with $n\geq 2$, and consider the following reiterated Dirichlet problems for Stokes systems depending on a parameter $\varepsilon>0$,

\begin{equation}
\left\{
\begin{aligned}
\mathcal{L}_{\varepsilon} u_{\varepsilon} +\nabla p_\varepsilon&=f  \text {\quad\ \ \  in } \Omega, \\
\operatorname{div}u_\varepsilon&=h \text {\quad\ \ \  in } \Omega,\\
u_{\varepsilon}&=g  \text {\quad\ \ \  on } \partial \Omega,
\end{aligned}\right.
\end{equation}
with the compatibility condition \begin{equation}
\int_\Omega h-\int_{\partial\Omega}g\cdot n=0,
\end{equation} where $n$ denotes the outward unit normal to $\partial\Omega$. Here $\varepsilon>0$ is a small parameter and the operator $\mathcal{L}_\varepsilon$ is defined by
\begin{equation}
\mathcal{L}_{\varepsilon}=-\operatorname{div}(A(x / \varepsilon,x/\varepsilon^2) \nabla)=-\frac{\partial}{\partial x_{i}}\left[a_{i j}^{\alpha \beta}\left(\frac{x}{\varepsilon},\frac{x}{\varepsilon^2}\right) \frac{\partial}{\partial x_{j}}\right]
\end{equation} with $1\leq i,j,\alpha,\beta\leq n$.\\

Given constants $\mu>0$, and $M>0$ such that the coefficient matrix $A(y,z)=(a_{ij}^{\alpha\beta}(y,z))$ is real, bounded measurable, and satisfies the following conditions.\\
$\bullet $ The ellipticity condition.
\begin{equation}
\mu|\xi|^{2} \leq a_{i j}^{\alpha \beta}(y) \xi_{i}^{\alpha} \xi_{j}^{\beta} \leq \frac{1}{\mu}|\xi|^{2} \quad \text { for } y,z \in \mathbb{R}^{n} \text { and } \xi=\left(\xi_{i}^{\alpha}\right) \in \mathbb{R}^{n \times n},
\end{equation}
$\bullet $ The smoothness condition. There exist a constant $M>0$, such that for any $y_1,y_2,z\in \mathbb{R}^n$, there holds
\begin{equation}\label{1.3}|A(y_1,z)-A(y_2,z)|\leq M|y_1-y_2|.\end{equation}
$\bullet $ The periodicity condition.
\begin{equation}\label{1.4}A(y,z) \text { is } Y-Z \text { periodic.}\end{equation}

For simplicity, we may assume $Y=Z=(0,1)^n$. From the asymptotic expansion, we can obtain the following correctors for the reiterated Stokes system,
\begin{equation}\left\{
\begin{aligned}
\mathcal{A}_1(\chi_k^\beta(y,z)-P_k^\beta(z))+\nabla_z\pi_k^\beta(y,z)&=0 \text{\quad in }Z,\\
\operatorname{div}_z\chi_k^\beta(y,z)&=0\text{\quad in }Z,\\
\fint_Z\chi_k^\beta(y,z)dz=0,\ \ \fint_Z\pi_k^\beta(y,z)&=0,
\end{aligned}\right.\end{equation}
where \begin{equation}(\mathcal{A}_1u)^\alpha= -\frac{\partial}{\partial z_{i}}\left(a_{i j}^{\alpha\beta}(y, z) \frac{\partial u^\beta}{\partial z_{j}}\right),
\end{equation} $P_j^\beta(y)=y_je^\beta=y_j(0,\cdots,1,\cdots,0)$ with 1 in the $\beta$-th position, and
\begin{equation}\left\{
\begin{aligned}
\mathcal{A}_2(\chi_k^\beta(y)-P_k^\beta(y))+\nabla_y\pi_k^\beta(y)&=0 \text{\quad in }Y,\\
\operatorname{div}_y\chi_k^\beta(y)&=0\text{\quad in }Y,\\
\fint_Y\chi_k^\beta(y)dy=0,\ \ \fint_Y\pi_k^\beta(y)&=0,
\end{aligned}\right.\end{equation}
where \begin{equation}\begin{aligned}(\mathcal{A}_2u)^\alpha=& -\frac{\partial}{\partial y_{i}}\left[a_{2,ij}^{\alpha\beta}(y)\frac{\partial u^\beta}{\partial y_j}\right]\\
=& -\frac{\partial}{\partial y_{i}}\left[\left(\fint_Z\left(a_{i j}^{\alpha\beta}(y,z)-a_{i k}^{\alpha\gamma}(y,z) \frac{\partial \chi_j^{\gamma\beta}(y,z)}{\partial z_{k}}\right)dz\right)\frac{\partial u^\beta}{\partial y_j}\right],
\end{aligned}\end{equation} with $1\leq i,j,k,\alpha,\beta,\gamma\leq n$. Consequently, the homogenized equation is
\begin{equation}
\left\{
\begin{aligned}
\mathcal{L}_{0} u_{0} +\nabla p_0&=f  \text {\quad in } \Omega, \\
\operatorname{div}u_0&=h \text {\quad in } \Omega,\\
u_{\varepsilon}&=g  \text {\quad on } \partial \Omega,
\end{aligned}\right.
\end{equation} with $$(\mathcal{L}_0u_0)^\beta=-\frac{\partial}{\partial x_{i}}\left(\widehat{a}_{i j}^{\alpha \beta} \frac{\partial u^\beta}{\partial x_{j}}\right),$$
where the operator $\widehat{A}=(\widehat{a}_{ij}^{\alpha\beta})$ is a constant matrix defined as
\begin{equation}\label{1.6}
\hat{a}_{i j}^{\alpha\beta}=\frac{1}{|Y||Z| }\iint_{Y \times Z}\left[a_{i j}^{\alpha\beta}-a_{i k}^{\alpha\gamma} \frac{\partial \chi_{j}^{\gamma\beta}(y,z)}{\partial z_{k}}-a_{i k}^{\alpha\gamma} \frac{\partial \chi_{j}^{\gamma\beta}(y)}{\partial y_{k}}+a_{i k}^{\alpha\gamma} \frac{\partial \chi_{l}^{\gamma\sigma}(y,z)}{\partial z_{k}} \frac{\partial \chi_{j}^{\sigma\beta}(y)}{\partial y_{l}}\right] d y d z.
\end{equation}

Note that due to $a_{ij}(y,z)$ is Y-Z periodic, then the solution $(\chi_k(y,z),\pi_k(y,z)$ of the equation $(1.7)$ is also Y-Z periodic, which is useful for the Fourier transform methods.
Throughout this paper, we use the following notation
$$H^m_{\text{per}(Y)}=:\left\{f\in H^m(Y) \text{ and }f\text{ is Y-periodic with }\fint_Yfdy=0\right\}. $$

The following theorem is the main result of the paper, which establishes the $O(\varepsilon)$ convergence rates in $L^2(\Omega)$ for the Dirichlet problems.
\begin{thm} (convergence rates for the velocity term). Let $\Omega\subset \mathbb{R}^n$ be a bounded $C^{1,1}$ domain, and assume that $A(y,z)$ satisfies the
conditions $(1.4),\ (\ref{1.3})$ and $(\ref{1.4})$. Given $h\in H^1(\Omega),$ and $g\in H^{3/2}(\partial \Omega;\mathbb{R}^n)$ satisfying the compatibility condition $(1.2)$, for $f\in L^2(\Omega;\mathbb{R}^n)$, let $(u_\varepsilon,p_\varepsilon)$, $(u_0,p_0)$ be the weak solutions of $(1.1)$ and $(1.11)$,  respectively. Then there holds the following estimates \begin{equation*}
||u_\varepsilon-u_0||_{L^2(\Omega)}\leq C\varepsilon||u_0||_{H^2(\Omega)},
\end{equation*} where $C$ depends on $\mu,n$ and $\Omega$.\end{thm}
In this paper, we also obtain $O(\varepsilon^{1/2})$ rates for the pressure term $p_\varepsilon$, which is stated in the following Theorem.
\begin{thm} (convergence rates for the pressure term). Let $\Omega\subset \mathbb{R}^n$ be a bounded $C^{1,1}$ domain, and assume that $A(y,z)$ satisfies the
conditions $(1.4),\ (\ref{1.3})$ and $(\ref{1.4})$. Given $h\in H^1(\Omega),$ and $g\in H^{3/2}(\partial \Omega;\mathbb{R}^n)$ satisfying the compatibility condition $(1.2)$, for $f\in L^2(\Omega;\mathbb{R}^n)$, let $(u_\varepsilon,p_\varepsilon)$, $(u_0,p_0)$ be the weak solutions of $(1.1)$ and $(1.11)$,  respectively. Moreover, if $\int_\Omega p_\varepsilon=\int_\Omega p_0=0$, then there holds the following estimates
\begin{equation}||p_\varepsilon-p_0+\tilde{\pi}-\int_\Omega\tilde{\pi}||_{L^2(\Omega)}\leq C\varepsilon^{1/2}||u_0||_{H^2(\Omega)},\end{equation}
where $\tilde{\pi}=\pi_k^\beta(y,z)\psi_{2\varepsilon}S_\varepsilon(\partial_{k}u_0^\beta)+\pi_k^\beta(y)\psi_{2\varepsilon}S_\varepsilon(\partial_{k}u_0^\beta)
-\pi_j^\gamma(y,z)\partial_{y_j}\chi_k^{\gamma\beta}(y)\psi_{2\varepsilon}S_\varepsilon(\partial_{k}u_0^\beta)$ with $y=x/\varepsilon$, and $z=x/\varepsilon^2$, in which $\psi_\varepsilon$ is a cut-off function defined in $(3.1)$ and $S_\varepsilon$ is the smoothing operator defined in $(2.26)$ and $C$ depends on $\mu,n$ and $\Omega$.\end{thm}

The convergence rate is one of the central issues in homogenization theory and has been studied extensively in the various setting. For elliptic equations and systems in divergence form with periodic coefficients, related results may be found in the recent work \cite{shen2018periodic,Kenig2010Homogenization,shen2017boundary,kenig2012convergence}.

For the homogenization of Stokes systems problems, the authors in \cite{Shu2015Homogenization} have established the interior Lipschitz estimates for
the velocity and $L^\infty$ estimates for the pressure as well as the $W^{1,p}$ estimates in a bounded $C^1$ domain for any $1<p<\infty$ under the smoothness condition: the coefficients matrix $A(x)\in VMO(\mathbb{R}^n)$.

Gu \cite{Shu2015Convergence} has obtained the following sharp $O(\varepsilon)$ error estimates:
$$||u_\varepsilon-u_0||_{L^2(\Omega)}\leq C \varepsilon||u_0||_{H^2(\Omega)},$$
as well as the $O(\varepsilon^{1/2})$ error estimates for the pressure term with $\Omega$ a bounded $C^{1,1}$ domain for the Stokes systems problems. Later, Xu \cite{xu2017convergence} generalizes this problem to Lipschitz domain, and has obtained the error estimates as well as the $W^{1,p}$ estimates, where $|\frac{1}{p}-\frac{1}{2}|<\frac{1}{2n}+\epsilon$ and $\epsilon$ is a positive constant independent of $\varepsilon$.

In this paper, our aim is to obtain the error estimates for the reiterated Stokes systems problems. In order to separate the different scale of $x$ we use the Fourier transform methods which was first introduced by Zhang in \cite{Zhang1909} and obtain the $O(\varepsilon)$ error estimates for the velocity term and $O(\varepsilon^{1/2})$ error estimates for the pressure term.

\section{Preliminaries}
\subsection{basic results}
In this subsection, we introduce the definition of weak solution to the equation $(1.1)$ and recall some basic results which are useful for the correctors estimates.
\begin{defn} We say that $(u_\varepsilon,p_\varepsilon)\in H^1(\Omega;\mathbb{R}^n)\times L^2(\Omega)$ is a weak solution to $(1.1)$, if
$(u_\varepsilon,p_\varepsilon)$ satisfies
\begin{equation}
\begin{array}{l}{\text { 1. } B_{\varepsilon}\left[u_{\varepsilon}, \phi\right]-\int_{\Omega} p_{\varepsilon} \operatorname{div}(\phi) d x=\langle f, \phi\rangle_{H^{-1}(\Omega) \times H_{0}^{1}(\Omega)} \text { for any } \phi \in H_{0}^{1}\left(\Omega ; \mathbb{R}^{d}\right)} \\ {\text { 2. } \operatorname{div}\left(u_{\varepsilon}\right)=h \text { in the distribution sense in } \Omega,} \\ {\text { 3. } u_{\varepsilon}=g \text { in the trace sense on } \partial \Omega}, \end{array}
\end{equation}
where $B_{\varepsilon}[\cdot, \cdot]$  is the bilinear form defined by $$ B_{\varepsilon}[v, w]=\int_{\Omega} a_{i j}^{\alpha \beta}\left(\frac{x}{\varepsilon},\frac{x}{\varepsilon^2}\right) \frac{\partial v^{\beta}}{\partial x_{j}} \frac{\partial w^{\alpha}}{\partial x_{i}} d x \quad \text { for any } w, v \in H^{1}\left(\Omega ; \mathbb{R}^{n}\right).$$
\end{defn}
We firstly introduce the following lemma whose proof may be found in \cite{xu2017convergence}.
\begin{lemma} Suppose $A$ satisfies $(1.4)$. Let $f\in H^{-1}(\Omega;\mathbb{R}^n)$, $h\in L^2(\Omega)$ and $g\in H^{1/2}(\partial\Omega;\mathbb{R}^n)$ with the compatibility condition $(1.2)$. Then the Dirichlet problem $(1.1)$ has a unique weak solution $(u_\varepsilon,p_\varepsilon)\in H^1(\Omega;\mathbb{R}^n)\times L^2(\Omega)$, with $p_\varepsilon$ unique up to constants, and we have the following uniform estimates
\begin{equation}
\left\|u_{\varepsilon}\right\|_{H^{1}(\Omega)}+||p_{\varepsilon}-\fint_\Omega p_\varepsilon||_{L^{2}(\Omega)} \leq C\left\{\|f\|_{H^{-1}(\Omega)}+\|h\|_{L^{2}(\Omega)}+\|g\|_{H^{1/2}(\partial \Omega)}\right\},
\end{equation} where C depends only on $n,\mu$ and $\Omega$.
\end{lemma}

We also need the following lemma which states the interior $W^{1,p}$ estimates if the coefficient $A(y)\in VMO(\mathbb{R}^n)$:
\begin{equation}
\sup _{y \in \mathbb{R}^{d}} \fint_{B(y, t)}\left|A-\fint_{B(y, t)} A\right| \leq \omega_{1}(t),
\end{equation} where $w$ is a fixed nondecreasing continuous function on $[0,\infty)$ with $w(0)=0.$ For the proof of Lemma 2.3, see \cite{Shu2015Homogenization} for example.
\begin{lemma}
Suppose that $A(y)$ satisfies the ellipticity condition $(1.4)$ and smoothness condition $(2.3)$. Let $(u,p)\in H^1(B(0,1);\mathbb{R}^n)\times L^2(B(0,1))$ be a weak solution to \begin{equation}
-\operatorname{div}(A(x) \nabla u)+\nabla p=f  \text { and } \operatorname{div}(u)=0 \end{equation} in $B(0,1)$, with $f\in L^q(B(0,1))$  for any  $2<q<\infty$. Then $|\nabla u|\in L^{q}(B(0,1 / 2)) $,  and  \begin{equation}\left(\fint_{B(0,1 / 2)}|\nabla u|^{q}\right)^{1 / q}\leq C(q,n)\left(\fint_{B(0,1)}|\nabla u|^{2}\right)^{1 / 2}+C(q,n)\left(\fint_{B(0,1)}|f|^{q}\right)^{1 / q}.
\end{equation}
\end{lemma}
\subsection{Correctors estimates}
In this subsection, we give some basic estimates for the correctors $\chi_k^\alpha(y,z)$ and $\chi_k^\alpha(y)$.
\begin{lemma}
Let $(\chi_k^\beta(y),\pi_k^\beta(y))$ and $(\chi_k^\beta(y,z),\pi_k^\beta(y,z))$ be the weak solution of $(1.9)$ and $(1.7)$, respectively. Then there hold
\begin{equation}\label{2.1}\fint_Z\left|\nabla_y\pi_k^\beta(y,z)\right|^2dz+\fint_Z\left|\nabla_y\chi_k^\beta(y,z)\right|^2dz+\fint_Z\left|\nabla_z\nabla_y\chi_k^\beta(y,z)\right|^2dz\leq C_1\end{equation} and
\begin{equation}\label{2.2}||\chi_k^\beta(y)||_{W^{2,p}(Y)}+||\pi_k^\beta(y)||_{W^{1,p}(Y)}\leq C_2\end{equation} for any $p\in(1,\infty)$ and $k,\beta=1,2,\cdots,n$, where $C_1$ depends on $\mu,M$ and $n$; and $C_2$ depends on $\mu,p,M$ and $n$.
\end{lemma}
\begin{proof}The proof is standard. Firstly, testing the equation $(1.7)$ with $\chi_k^\beta(y,z)$ gives that
\begin{equation}\label{2.3}||\chi_k^\beta(y,\cdot)||_{W^{1,2}(Z)}\leq C.\end{equation}
Then for any $y_1,y_2$, there holds
\begin{equation}\left\{\begin{aligned}&-\partial_{z_i}\left(a_{ij}^{\alpha\gamma}(y_2,z)\partial_{z_j}(\chi_k^{\gamma\beta}(y_1,z)-\chi_k^{\gamma\beta}(y_2,z))\right)+
\nabla_{z_\alpha}(\pi_k^\beta(y_1,z)-\pi_k^\beta(y_2,z))\\
=&\partial_{z_i}\left((a_{ij}^{\alpha\gamma}(y_1,z)-a_{ij}^{\alpha\gamma}(y_2,z))\partial_{z_j}(\chi_k^\beta(y_1,z)-P_k^\beta(z))^\gamma\right),\\
&\operatorname{div}_z(\chi_k^\beta(y_1,z)-\chi_k^\beta(y_2,z))=0,
\end{aligned}\right.\end{equation} then, according to Lemma 2.2, we have
\begin{equation}\begin{aligned}&\fint_Z\left|\nabla_z(\chi_k^{\beta}(y_1,z)-\chi_k^{\beta}(y_2,z))\right|^2dz+
\fint_Z\left|\pi_k^\beta(y_1,z)-\pi_k^\beta(y_2,z)\right|^2dz\\
\leq &C\fint_Z|A(y_1,z)-A(y_2,z)|^2(1+|\nabla_z(\chi_k^\beta(y_1,z)|^2)dz\\
\leq &C|y_1-y_2|^2,
\end{aligned}\end{equation}
due to $(\ref{1.3})$ and $(\ref{2.3})$, thus this together with Poinc\'{a}re inequality will give the state estimate $(\ref{2.1})$.

Note that
$$\begin{aligned}|\nabla_ya_{2,ij}^{\alpha\beta}(y)|&\leq\fint_Z\left(|\nabla_yA(y,z)|+|\nabla_yA(y,z)|| \nabla_{z} \chi_{j}^{\beta}(y,z)|+A(y,z) |\nabla_y\nabla_{z} \chi_{j}^{\beta}(y,z)|\right)dz\\
&\leq C,\end{aligned}$$ with $a_{2,ij}^{\alpha\beta}$ defined in $(1.10)$,
where we have used $(\ref{1.3})$, $(\ref{2.1})$ and$(\ref{2.3})$. Then $(1.9)$ and Lemma 2.3 as well as $\chi_k^\beta(y)$ is Y-periodic yields
$||\chi_k^\beta(y)||_{W^{1,p}(Y)}\leq C.$ Note that $\fint_Y\pi_k^\beta(y)dy=0$, then $||\pi_k^\beta(y)||_{L^p(Y)}\leq C$ follows from the first line of $(1.9)$ and $||\nabla_y \chi_k^\beta(y)||_{L^p(Y)}\leq C$.
To complete the proof of $(2.7)$, we just need to take the derivative of $y$ with respect to the equation $(1.9)$, and we can obtain the desired estimate $(2.7)$ according to Lemma $(2.3)$ again.
\end{proof}

\begin{lemma}(reverse H\"{o}lder inequality). Let $(\chi_k^\beta(y,z),\pi_k^\beta(y,z))$ be the weak solution to $(1.7)$, then there exists a constant $\tau>0$ which depends on $\mu$ and $n$, such that for any $y$, there holds
\begin{equation}\label{2.7}
||\nabla_z \chi_k^\beta(y,z)||_{L^{2+\tau}(Z)}+||\nabla_z \pi_k^\beta(y,z)||_{L^{2+\tau}(Z)}\leq C,
\end{equation} where $C$ depends on $\mu$ and $n$.
\end{lemma}
\begin{proof} Recall that $\fint_Z\pi_k^\beta(y,z)dz=0$, then
 $||\pi_k^\beta(y,z)||_{L^{2+\tau}(Z)}\leq C$ follows form  the first line of $(1.7)$ and  $||\nabla_z\chi_k^\beta(y,z)||_{L^{2+\tau}(Z)}\leq C$. Consequently, we need only to prove that $||\nabla_z \chi_k^\beta(y,z)||_{L^{2+\tau}(Z)}\leq C.$
Then for any $z_0\in Z$, the Caccioppoli's inequality gives that
\begin{equation}
\int_{B}\left|\nabla_z \chi_{k}^{\beta}(y,z)\right|^{2} d z \leq \frac{C}{r^{2}}\left\{\int_{2 B }\left|\chi_{k}^{\beta}(y,z)-c\right|^{2} d z+r^{n+2}\right\},
\end{equation}where $B=B(z_0,r)$, and for any $c\in\mathbb{R}^n$.
Then, choose $c=\fint_{2B}\chi_k^\beta(y,z)dz$ and the Sobolev-Poinc\'{a}re inequality leads to
\begin{equation}\label{2.9}\fint_{B}|\nabla_z \chi_k^\beta(y,z)|^2dz\leq C\left(\fint_{2B}|\nabla_z \chi_k^\beta(y,z)|^{\frac{2n}{2+n}}dz\right)^{\frac{2+n}{n}}+C.\end{equation}
Using the reverse inequality (see \cite[Chapter V, Theorem 1.2]{Giaquinta1983Multiply}), we could obtain higher integrability, and there exists a $\tau>0$,
depending on $\mu,n$ such that
\begin{equation}\label{2.10}\fint_{B}|\nabla_z \chi_k^\beta(y,z)|^{2+\tau}dz\leq C\left(\fint_{2B}|\nabla_z \chi_k^\beta(y,z)|^2dz\right)^{\frac{2+\tau}{2}}+C.\end{equation}

 Consequently, a covering argument will lead to the desired estimate $(2.11)$ due to $||\chi_k^\beta(y,z)||_{W^{1,2}(Z)}\leq C$ and $\chi_k^\beta(y,z)$ is Z-periodic.
\end{proof}
In the following three lemmas, we introduce three flux correctors which will be useful for obtaining the convergence rates.
\begin{lemma}Let
\begin{equation}\begin{aligned}I_{1,ij}^{\alpha\beta}(y,z)\triangleq&-a_{ij}^{\alpha\beta}(y,z)+a_{ik}^{\alpha\gamma}(y,z)\partial_{z_k}\chi_j^{\gamma\beta}(y,z)\\
&+\fint_Z\left(a_{ij}(y,z)-a_{ik}^{\alpha\gamma}(y,z)\partial_{z_k}\chi_j^{\gamma\beta}(y,z)\right)dz,\end{aligned}\end{equation}
where $y\in Y$ and $z\in Z$. Then there hold: $(i)$ $\fint_ZI_{1,ij}^{\alpha\beta}(y,\cdot)dz=0$; $(ii)$ $\partial_{z_i}I_{1,ij}^{\alpha\beta}=\nabla_{z_\alpha}\pi_j^\beta(y,z)$. Moreover,
there exist  $E_{1,kij}^{\alpha\beta}(y,\cdot)\in H^1_{\text{per}}(Z)$ and $q_{1,ik}^\beta(y,\cdot)\in H^1_{\text{per}}(Z)$ such that
\begin{equation}\label{2.12}
I_{1,ij}^{\alpha\beta}(y,z)=\partial_{z_k}E_{1,kij}^{\alpha\beta}(y,z)+\partial_{z_\alpha}q_{1,ij}^\beta(y,z) \ ,\  E_{1,kij}^{\alpha\beta}=-E_{1,ikj}^{\alpha\beta}\text{ and }\partial_{z_i}q_{1,ik}^\alpha=\pi_k^\alpha(y,z),
\end{equation}
and there hold  the following estimates
\begin{equation}\label{2.13}\fint_Z|E_{1,kij}^{\alpha\beta}(y',z)-E_{1,kij}^{\alpha\beta}(y,z)|^2dz+\fint_Z|\nabla_z(E_{1,kij}^{\alpha\beta}(y',z)-E_{1,kij}^{\alpha\beta}(y,z))|^2dz\leq C|y-y'|^2\end{equation}
for any $k,i,j,\alpha,\beta=1,\cdots,n$, where $C$ depends on $\mu,M$ and $n$.\end{lemma}
\begin{proof}The $(i)$ and $(ii)$ follow from the definition $(2.15)$ and $(1.7)$, respectively.
Let $I_{1,ij}^{\gamma}(y,z)=(I_{1,ij}^{1\gamma}(y,z),\cdots,I_{1,ij}^{n\beta}(y,z))$, and we construct the auxiliary cell problem as follows
\begin{equation}
\left\{\begin{aligned} \Delta_z f_{1,i k}^{\gamma}(y,z)+\nabla_z q_{1,i k}^{\gamma}(y,z) &=I_{1,i k}^{\gamma}(y,z) \quad \text { in } Z \\ \operatorname{div}\left(f_{1,i k}^{\gamma}\right) &=0 \quad\quad\quad\quad\  \text { in } Z \\ \int_{Z} f_{1,i k}^{\gamma}(y,z) d z =0, \int_Z q_{1,i k}^{\gamma}(y,z) dz&=0, \text { and } f_{1,i k}^{\gamma}(y,z), q_{1,i k}^{\gamma}(y,z)\quad \text { are Z-periodic, } \end{aligned}\right.
\end{equation} The existence of the solution $(f_{1,ik}^\gamma,q_{1,ik}^\gamma)\in H_{\text{per}}^2(Z;\mathbb{R}^n)\times H^1_{\text{per}}(Z)$ to the equation $(2.18)$ is based upon the property (i), $I_{1,ij}^{\alpha\gamma}(y,\cdot)\in L^2(Z)$ and Lemma 2.2.
Set $E_{1,kij}^{\alpha\gamma}(y,\cdot)=\partial_{z_k}f_{1,ij}^{\alpha\gamma}(y,\cdot)-\partial_{z_i}f_{1,kj}^{\alpha\gamma}(y,\cdot)$, then
$E_{1,kij}^{\alpha\gamma}=-E_{1,ikj}^{\alpha\gamma}$ is clear, and we find $$\partial_{z_k}E_{1,kij}^{\alpha\gamma}(y,z)=\Delta_zf_{1,ij}^{\alpha\gamma}(y,z)-\partial_{z_i}\partial_{z_k}f_{1,kj}^{\alpha\gamma}(y,z)=
I_{1,ij}^{\alpha\gamma}(y,z)-\partial_{z_\alpha}q_{1,ij}^{\alpha\gamma}(y,z)-\partial_{z_i}\partial_{z_k}f_{1,kj}^{\alpha\gamma}(y,z).$$
Consequently, it only needs to prove $\partial_{z_i}\partial_{z_k}f_{1,kj}^{\alpha\gamma}(y,z)=0$.
In view of $(2.18)$ and the property $(ii)$, we have
\begin{equation}
\left\{\begin{aligned} \Delta_z\left(\frac{\partial f_{1,i k}^{\alpha \gamma}}{\partial z_{i}}\right)+\nabla_{z_\alpha}\left(\frac{\partial q_{1,i k}^{\gamma}}{\partial z_{i}}\right)=\frac{\partial I_{1,i k}^{\alpha \gamma}}{\partial z_{i}} &=\nabla_{\alpha} \pi_{k}^{\gamma} \quad \text { in } Z \\ \nabla_{\alpha}\left(\frac{\partial f_{1,i k}^{\alpha \gamma}}{\partial z_{i}}\right)&=0  \text { in } Z \end{aligned}\right.
\end{equation}

This implies $\partial_{z_i}f_{1,ij}^{\alpha\gamma}$ is a constant, (taking $\partial_{z_i}f_{1,ij}^{\alpha\gamma}$ as a test function and integrating by
parts, it is not hard to derive $\int_z|\nabla_z(\partial_{z_i}f_{1,ij}^{\alpha\gamma})|^2dz=0$,) therefore, we have
$\partial_{z_i}\partial_{z_k}f_{1,kj}^{\alpha\beta}=0$. Also, the above equation shows the difference between $\pi_k^\alpha(y,z)$ and
$\partial_{z_i}q_{1,ik}^\alpha(y,z)$ is a constant. Consequently, $\partial_{z_i}q_{1,ik}^\alpha=\pi_k^\alpha(y,z)$ follows from the facts that
$\int_Z\pi_k^\alpha(y,z)dz=\int_Z\partial_{z_i}q_{1,ik}^\alpha(y,z)dz=0$.\\

To prove the estimate $(2.17)$, we note that
$$\begin{aligned}
\int_{Z}|\nabla_{z}E_{1,kij}(y,\cdot)-\nabla_{z}E_{1,kij}(y',\cdot)|^2dz&\leq \int_{Z}|\nabla^2_{z}(f_{1,ij}(y,\cdot)-f_{1,ij}(y',\cdot))|^2dz\\
&\leq C\int_{2Z}|I_{1,ij}(y,z)-I_{1,ij}(y',z)|^2dz\\
&\leq C|y-y'|^2,
\end{aligned}$$ where we have used $(\ref{1.3})$ and $(2.10)$ in the last inequality. Consequently,
the estimate above together with Poincar\'{e}' inequality completes the proof of $(\ref{2.13})$.

\end{proof}

\begin{lemma}Let
\begin{equation}\label{2.14}\begin{aligned}
I_{2,ij}^{\alpha\beta}(y)\triangleq\widehat{a}_{ij}^{\alpha\beta}&+\fint_Z\left(a_{ik}^{\alpha\gamma}(y,z)\partial_{y_k}\chi_j^{\gamma\beta}(y)
-a_{ik}^{\alpha\gamma}(y,z)\partial_{z_k}\chi_l^{\gamma\eta}(y,z)\partial_{y_l}\chi_j^{\eta\beta}(y)\right)dz\\
&-\fint_Z\left(a_{ij}^{\alpha\beta}(y,z)-a_{ik}^{\alpha\gamma}(y,z)\partial_{z_k}\chi_j^{\gamma\beta}(y,z)\right)dz,
\end{aligned}\end{equation}\end{lemma}
where $y\in Y$, we assume that $\fint_YI_{2,ij}^{\alpha\beta}(y)dy=0$ in addition. Then  $\partial_{y_i}I_{2,ij}^{\alpha\beta}=\nabla_{y_\alpha}\pi_j^\beta(y)$.
Moreover, there exists  $E_{2,kij}^{\alpha\beta}(y)\in H^1_{\text{per}}(Y)$ and $q_{2,ik}^\beta(y)\in H^1_{\text{per}}(Y)$ such that
\begin{equation}
I_{2,ij}^{\alpha\beta}(y)=\partial_{y_k}E^{\alpha\beta}_{2,kij}(y)+\partial_{y_\alpha}q_{2,ik}^\beta(y) \ ,\  E_{2,kij}^{\alpha\beta}=-E_{2,ikj}^{\alpha\beta}\text{ and }\partial_{y_i}q_{2,ik}^\alpha(y)=\pi_k^\alpha(y),
\end{equation}
and the estimate \begin{equation}\label{2.16}
||E^{\alpha\beta}_{2,kij}||_{H^1(Y)}\leq C,
\end{equation}
for any $k,i,j,\alpha,\beta=1,\cdots,n$, where $C$ depends on $\mu,M$ and $n$.
\begin{proof}The proof is totally similarly to Lemma 2.6.
\end{proof}

\begin{lemma}Let
\begin{equation}\label{2.17}\begin{aligned}
I_{3,ij}^{\alpha\beta}(y,z)\triangleq&a_{ik}^{\alpha\gamma}(y,z)\partial_{y_k}\chi_j^{\gamma\beta}(y)-a_{ik}^{\alpha\eta}(y,z)\partial_{z_k}\chi_l^{\eta\gamma}(y,z)\partial_{y_l}\chi_j^{\gamma\beta}(y)\\
&-\fint_Z\left(a_{ik}^{\alpha\gamma}(y,z)\partial_{y_k}\chi_j^{\gamma\beta}(y)-a_{ik}^{\alpha\gamma}(y,z)\partial_{z_k}\chi_l^{\gamma\eta}(y,z)\partial_{y_l}\chi_j^{\eta\beta}(y)\right)dz,
\end{aligned}\end{equation}
where $y\in Y$ and $z\in Z$. Then there hold: $(i)$ $\fint_ZI_{3,ij}^{\alpha\beta}(y,\cdot)dz=0$; $(ii)$ $\partial_{z_i}I_{3,ij}^{\alpha\beta}(y,z)=-\partial_{z_\alpha}(\pi_k^\gamma(y,z)\partial_{y_k}\chi_j^{\gamma\beta}(y))$.
Moreover, there exists  $E_{3,kij}^{\alpha\beta}(y,\cdot)\in H^1_{\text{per}}(Z)$ and $q_{3,ik}^\beta(y,\cdot)\in H^1_{\text{per}}(Z)$ such that
\begin{equation}\begin{aligned}
I_{3,ij}^{\alpha\beta}(y,z)&=\partial_{z_k}E^{\alpha\beta}_{3,kij}(y,z)-\partial_{z_\alpha}q_{3,ik}^\beta(y,z), \\ E^{\alpha\beta}_{3,kij}&=-E_{3,ikj}^{\alpha\beta}\text{ and }\partial_{z_i}q_{3,ik}^\alpha(y,z)=-\pi_j^\gamma(y,z)\partial_{y_j}\chi_k^{\gamma\alpha}(y).
\end{aligned}\end{equation}
for any $i,j,\alpha,\beta=1,\cdots,n$.\end{lemma}
\begin{proof}The $(i)$ and $(ii)$ follow from the definition $(2.23)$ and $(1.7)$, respectively. Actually, similarly to the proof of Lemma 2.6, the existences of $E_{3,kij}^{\alpha\gamma}(y,z)$ and $q_{3,ik}^\gamma(y,z)$ are given by the following the auxiliary cell problem
\begin{equation}
\left\{\begin{aligned} \Delta_z f_{3,i k}^{\gamma}(y,z)+\nabla_z q_{3,i k}^{\gamma}(y,z) &=I_{3,i k}^{\gamma}(y,z) \quad \text { in } Z \\ \operatorname{div}\left(f_{3,i k}^{\gamma}\right) &=0 \quad \text { in } Z \\ \int_{Z} f_{3,i k}^{\gamma}(y,z) d z =0, \int_Z q_{3,i k}^{\gamma}(y,z) dz&=0, \text { and } f_{3,i k}^{\gamma}(y,z),\ q_{3,i k}^{\gamma}(y,z) \text { are Z-periodic, } \end{aligned}\right.
\end{equation} with $E_{3,kij}^{\alpha\gamma}(y,\cdot)=\partial_{z_k}f_{3,ij}^{\alpha\gamma}(y,\cdot)-\partial_{z_i}f_{3,kj}^{\alpha\gamma}(y,\cdot)$.
\end{proof}
\subsection{Smoothing operator}
To deal with the convergence rates in the next section, we introduce an $\varepsilon$-smoothing operator $S_\varepsilon$ in this subsection.\\
\begin{defn} Fix a nonnegative function $\rho\in C_0^\infty(B(0,1/2))$ such that $\int_{\mathbb{R}^n}\rho dx=1$. For $\varepsilon>0,$
define \begin{equation}\label{2.19}
S_\varepsilon(f)(x)=\rho_\varepsilon\ast f(x)=\int_{\mathbb{R}^n}f(x-y)\rho_\varepsilon(y)dy,\end{equation}
where $\rho_\varepsilon(y)=\varepsilon^{-n}\rho(y/\varepsilon)$. \end{defn}
\begin{lemma} (i) If $f\in L^p(\mathbb{R}^n)$ for some $1\leq p<\infty.$ Then for any $g\in L^{p}_{\text{per}}(\mathbb{R}^n)$ ($g$ is $Y$-periodic),
\begin{equation}\label{2.20}\begin{cases}
||g(\cdot/\varepsilon)S_\varepsilon(f)||_{L^p(\mathbb{R}^n)}\leq C(p,n)||g||_{L^p(Y)}||f||_{L^p(\mathbb{R}^n)}\\
||g(\cdot/\varepsilon)\nabla S_\varepsilon(f)||_{L^p(\mathbb{R}^n)}\leq C(p,n)\varepsilon^{-1}||g||_{L^p(Y)}||f||_{L^p(\mathbb{R}^n)},
\end{cases}\end{equation} and if $0<\varepsilon\leq 1$, we have
\begin{equation}\label{2.21}
||g(\cdot/\varepsilon^2)S_\varepsilon(f)||_{L^p(\mathbb{R}^n)}\leq C(p,n)||g||_{L^p(Y)}||f||_{L^p(\mathbb{R}^n)}.
\end{equation}
(ii) If $f\in W^{1,p}(\mathbb{R}^n)$ for some $1\leq p<\infty.$ Then
\begin{equation}\label{2.22}
||S_\varepsilon(f)-f||_{L^p(\mathbb{R}^n)}\leq C(n,p)\varepsilon||\nabla f||_{L^p(\mathbb{R}^n)}.
\end{equation}
\end{lemma}

\begin{proof}For the proof of $(2.27)$, see for example \cite[Proposition 3.1.5]{shen2018periodic}, for the proof of (ii), see for example \cite[Proposition 3.1.6]{shen2018periodic}. Therefore, we need only give the proof of $(\ref{2.21})$. By H\"{o}lder's inequality,
$$\left|S_\varepsilon(f)(x)\right|^p\leq\int_{\mathbb{R}^n}\left|f(y)\right|^p\rho_{\varepsilon}(x-y)dy.$$
This together with Fubini's Theorem, gives
\begin{equation}\label{2.24}\begin{aligned}
\int_{\mathbb{R}^n}\left|g(x/\varepsilon^2)\right|^p\left|S_\varepsilon(f)(x)\right|^pdx&\leq\iint_{\mathbb{R}^n\times
\mathbb{R}^n}\left|g(x/\varepsilon^2)\right|^p\left|f(y)\right|^p\rho_{\varepsilon}(x-y)dx dy\\
&\leq C \sup_{y\in\mathbb{R}^n}\varepsilon^{-n}\int_{|x-y|\leq\varepsilon/2}\left|g(x/\varepsilon^2)\right|^pdx||f||_{L^p(\mathbb{R}^n)}^p\\
&\leq C \sup_{y\in\mathbb{R}^n}\varepsilon^{n}\int_{|x-y|\leq1/(2\varepsilon)}\left|g(x)\right|^pdx||f||_{L^p(\mathbb{R}^n)}^p\\
&\leq C ||f||_{L^p(\mathbb{R}^n)}^p||g||_{L^p(Y)}^p, \end{aligned}\end{equation}
where we use the periodicity of $g$ and note that $0<\varepsilon\leq1$.
\end{proof}
\begin{rmk}Actually, under the assumption of Lemma 2.7 (i), if $0<\varepsilon\leq 1$, for any $\lambda\geq \mu>0$, there holds
\begin{equation}\label{2.25}
||g(\cdot/\varepsilon^\lambda)S_{\varepsilon^\mu}(f)||_{L^p(\mathbb{R}^n)}\leq C(p,n)||g||_{L^p(Y)}||f||_{L^p(\mathbb{R}^n)}.
\end{equation} However, the similar results couldn't hold for the function $g(\cdot/\varepsilon^\lambda)S_{\varepsilon^\mu}(f)$,
if $0<\lambda< \mu$, unless the function $g$ has better regularity.
\end{rmk}
\section{Convergence rates}
First of all, we introduce the following cut-off function $\psi_r\in C_0^1(\Omega)$ associated with $\Sigma_r$:
\begin{equation}\label{3.1}
\psi_r=1\ \ \text{in\ }\Sigma_{2r},\ \ \ \psi_r=0\ \ \text{outside\ }\Sigma_{r},\ \ \ |\nabla\psi_r|\leq C/r,
\end{equation}
where $\Sigma_r=\{x\in\Omega:\text{dist}(x,\partial\Omega)>r\}.$

\begin{lemma}Suppose that $(u_\varepsilon,p_\varepsilon)$, $(u_0,p_0)\in H^1(\Omega;\mathbb{R}^n)\times L^2(\Omega)$ satisfy
\begin{equation}
\left\{\begin{array}{rlrl}{\mathcal{L}_{\varepsilon}\left(u_{\varepsilon}\right)+\nabla p_{\varepsilon}} & {=\mathcal{L}_{0}\left(u_{0}\right)+\nabla p_{0}} & {} & {\text { in } \Omega} \\ {\operatorname{div}\left(u_{\varepsilon}\right)} & {=\operatorname{div}\left(u_{0}\right)} & {} & {\text { in } \Omega} \\ {u_{\varepsilon}} & {=u_{0}} & {} & {\text { on } \partial \Omega}.\end{array}\right.
\end{equation} Let
\begin{equation}\label{3.2}\begin{aligned}
w_\varepsilon^\beta(x)=&u_\varepsilon^\beta(x)-u_0^\beta(x)+\varepsilon \chi_j^{\beta\gamma}(x/\varepsilon)\psi_{2\varepsilon}S_\varepsilon\left(\partial_{x_j}u_0^\gamma\right)\\
&+\varepsilon^2\chi_j^{\beta\alpha}(x/\varepsilon,x/\varepsilon^2)\left[\psi_{2\varepsilon}S_\varepsilon\left(\partial_{x_j}u_0^\alpha\right)
-\partial_{y_j}\chi_{k}^{\alpha\gamma}(x/\varepsilon)\psi_{2\varepsilon}S_\varepsilon\left(\partial_{x_k}u_0^\gamma\right)\right],
\end{aligned}\end{equation} Then we have
\begin{equation}
\left\{\begin{array}{rlrl}{\mathcal{L}_{\varepsilon}\left(w_{\varepsilon}\right)+\nabla\left(p_{\varepsilon}-p_{0}\right)} & {=-\operatorname{div}(f)} & {} & {\text { in } \Omega} \\ {\operatorname{div}\left(w_{\varepsilon}\right)} & {=\operatorname{div}\phi} & {} & {\text { in } \Omega} \\ {w_{\varepsilon}} & {=0} & {} & {\text { on } \partial \Omega},\end{array}\right.
\end{equation} and the compatibility condition
\begin{equation}\int_{\Omega}\operatorname{div}\phi(x)dx=0,\end{equation}
where \begin{equation}\phi^\beta=\varepsilon \chi_j^{\beta\gamma}(y)\psi_{2\varepsilon}S_\varepsilon\left(\partial_{x_j}u_0^\gamma\right)
+\varepsilon^2\chi_j^{\beta\alpha}(y,z)\left[\psi_{2\varepsilon}S_\varepsilon\left(\partial_{x_j}u_0^\alpha\right)
-\partial_{y_j}\chi_{k}^{\alpha\gamma}(y)\psi_{2\varepsilon}S_\varepsilon\left(\partial_{x_k}u_0^\gamma\right)\right],\end{equation}
with $y=x/\varepsilon$, $z=x/\varepsilon^2$, and $f=(f_i^\alpha)=(H_{1,i}^\alpha+H_{2,i}^\alpha+H_{3,i}^\alpha+H_{4,i}^\alpha)$  with $H_{j,i}^\alpha$ $(j=1,\cdots,4)$ defined in $(3.8)$.
\end{lemma}
\begin{proof}By direct computation,  we have
\begin{equation}\label{3.4}\begin{aligned}
&a_{ih}^{\alpha\beta}(x/\varepsilon,x/\varepsilon^2)\partial_hw_\varepsilon^{\beta}\\
=&a_{ih}^{\alpha\beta}(\partial_hu_\varepsilon^{\beta}-\partial_hu_0^{\beta})+a_{ih}^{\alpha\beta}\partial_{y_h}\chi_j^{\beta\gamma}(x/\varepsilon)\psi_{2\varepsilon}S_\varepsilon\left(\partial_{x_j}u_0^{\gamma}\right)
+\varepsilon a_{ih}^{\alpha\beta} \chi_j^{\beta\gamma}(x/\varepsilon)\partial_h\left(\psi_{2\varepsilon}S_\varepsilon\left(\partial_{x_j}u_0^{\gamma}\right)\right)\\
&+a_{ih}^{\alpha\beta}\partial_{z_h}\chi_j^{\beta\gamma}(x/\varepsilon,x/\varepsilon^2)\left[\psi_{2\varepsilon}S_\varepsilon\left(\partial_{x_j}u_0^{\gamma}\right)-\partial_{y_j}\chi_{k}^{\gamma\eta}(x/\varepsilon)
\psi_{2\varepsilon}S_\varepsilon\left(\partial_{x_k}u_0^{\eta}\right)\right]\\
&+\varepsilon a_{ih}^{\alpha\beta}\partial_{y_h}\chi_j^{\beta\gamma}(x/\varepsilon,x/\varepsilon^2)\left[\psi_{2\varepsilon}S_\varepsilon\left(\partial_{x_j}u_0^{\gamma}\right)-\partial_{y_j}\chi_{k}^{\gamma\eta}(x/\varepsilon)
\psi_{2\varepsilon}S_\varepsilon\left(\partial_{x_k}u_0^{\eta}\right)\right]\\
&+\varepsilon^2a_{ih}^{\alpha\beta}\chi_j^{\beta\gamma}(x/\varepsilon,x/\varepsilon^2)\left[\partial_h\left(\psi_{2\varepsilon}S_\varepsilon\left(\partial_{x_j}u_0^{\gamma}\right)\right)
-\partial_{y_j}\chi_{k}^{\gamma\eta}(x/\varepsilon)\partial_h\left(\psi_{2\varepsilon}S_\varepsilon\left(\partial_{x_k}u_0^{\eta}\right)\right)\right]\\
&-\varepsilon a_{ih}^{\alpha\beta}\chi_j^{\beta\gamma}(x/\varepsilon,x/\varepsilon^2)\partial^2_{y_jy_h}\chi_{k}^{\gamma\eta}(x/\varepsilon)\psi_{2\varepsilon}S_\varepsilon\left(\partial_{x_k}u_0^{\eta}\right)\\
=&:H_{0,i}^\alpha+H_{1,i}^\alpha+H_{2,i}^\alpha+H_{3,i}^\alpha+H_{4,i}^\alpha
\end{aligned}\end{equation}
and \begin{equation}\label{3.5}\begin{aligned}
H_{0,i}^\alpha=&a_{ih}^{\alpha\beta}\partial_hu_\varepsilon^{\beta}-\widehat{a}_{ih}^{\alpha\beta}\partial_hu_0^{\beta},\\
H_{1,i}^\alpha=&\left(\widehat{a}_{ih}^{\alpha\beta}-a_{ih}^{\alpha\beta})(\partial_hu_0^{\beta}-\psi_{2\varepsilon}S_\varepsilon\left(\partial_{x_h}u_0^{\beta}\right)\right)\\
H_{2,i}^\alpha=&[\widehat{a}_{ij}^{\alpha\beta}-a_{ij}^{\alpha\beta}+a_{ih}^{\alpha\gamma}\partial_{y_h}\chi_j^{\gamma\beta}(x/\varepsilon)+a_{ih}^{\alpha\gamma}\partial_{z_h}\chi_j^{\gamma\beta}(x/\varepsilon,x/\varepsilon^2)\\
& -a_{ih}^{\alpha\gamma}\partial_{z_h}\chi_k^{\gamma\eta}(x/\varepsilon,x/\varepsilon^2)\partial_{y_k}\chi_{j}^{\gamma\beta}(x/\varepsilon)]\psi_{2\varepsilon}S_\varepsilon\left(\partial_{j}u_0^{\beta}\right)\\
H_{3,i}^\alpha=&\varepsilon a_{ih}^{\alpha\beta} \chi_j^{\beta\gamma}(x/\varepsilon)\partial_h\left(\psi_{2\varepsilon}S_\varepsilon\left(\partial_{x_j}u_0^{\gamma}\right)\right)-\varepsilon a_{ih}^{\alpha\beta}\chi_j^{\beta\gamma}(x/\varepsilon,x/\varepsilon^2)\partial^2_{y_jy_h}\chi_{k}^{\gamma\eta}(x/\varepsilon)\psi_{2\varepsilon}S_\varepsilon\left(\partial_{x_k}u_0^{\eta}\right)\\
&+\varepsilon a_{ih}^{\alpha\beta}\partial_{y_h}\chi_j^{\beta\gamma}(x/\varepsilon,x/\varepsilon^2)\left[\psi_{2\varepsilon}S_\varepsilon\left(\partial_{x_j}u_0^{\gamma}\right)-\partial_{y_j}\chi_{k}^{\gamma\eta}(x/\varepsilon)
\psi_{2\varepsilon}S_\varepsilon\left(\partial_{x_k}u_0^{\eta}\right)\right],\\
H_{4,i}^\alpha=&\varepsilon^2a_{ih}^{\alpha\beta}\chi_j^{\beta\gamma}(x/\varepsilon,x/\varepsilon^2)\left[\partial_h\left(\psi_{2\varepsilon}S_\varepsilon\left(\partial_{x_j}u_0^{\gamma}\right)\right)
-\partial_{y_j}\chi_{k}^{\gamma\eta}(x/\varepsilon)\partial_h\left(\psi_{2\varepsilon}S_\varepsilon\left(\partial_{x_k}u_0^{\eta}\right)\right)\right].
\end{aligned}\end{equation}Consequently, according to $(3.2)$, we have
\begin{equation*}
\left\{\begin{array}{rlrl}{\mathcal{L}_{\varepsilon}\left(w_{\varepsilon}\right)+\nabla\left(p_{\varepsilon}-p_{0}\right)} & {=-\operatorname{div}(f)} & {} & {\text { in } \Omega} \\ {\operatorname{div}\left(w_{\varepsilon}\right)} & {=\operatorname{div}\phi} & {} & {\text { in } \Omega} \\ {w_{\varepsilon}} & {=0} & {} & {\text { on } \partial \Omega},\end{array}\right.
\end{equation*}
where $f_i^\alpha=H_{1,i}^\alpha+H_{2,i}^\alpha+H_{3,i}^\alpha+H_{4,i}^\alpha$. The compatibility condition $(3.5)$ is easy to verify since $\psi$ is a cut-off function.\end{proof}
\begin{lemma}Suppose that $A(y,z)$ satisfies $(1.4)-(1.6)$. Assume that $(u_\varepsilon,p_\varepsilon)$, $(u_0,p_0)$ are weak solutions to $(1.1)$ and $(1.11)$, respectively. Let $w_\varepsilon=(w_\varepsilon^\beta)$ with $w_\varepsilon^\beta$ defined in $(3.3)$ and
\begin{equation}\begin{aligned}
z_\varepsilon=&p_\varepsilon-p_0+\varepsilon^2\partial_{x_i}\left(q_{1,ik}^\beta(y,z)\psi_{2\varepsilon}S_\varepsilon(\partial_{k}u_0^\beta)\right)+\varepsilon \partial_{x_i}\left(q_{2,ik}^\beta(y)\psi_{2\varepsilon}S_\varepsilon(\partial_{k}u_0^\beta)\right)\\
&+\varepsilon^2\partial_{x_i}\left(q_{3,ik}^\beta(y,z)\psi_{2\varepsilon}S_\varepsilon(\partial_{k}u_0^\beta)\right),\\
& \text{with } y=x/\varepsilon,z=x/\varepsilon^2.
\end{aligned}\end{equation}
 Then $(w_\varepsilon, z_\varepsilon)$ satisfies \begin{equation}
\left\{\begin{array}{rlrl}{\mathcal{L}_{\varepsilon}\left(w_{\varepsilon}\right)+\nabla z_\varepsilon} & {=-\operatorname{div}(\tilde{f})} & {} & {\text { in } \Omega} \\ {\operatorname{div}\left(w_{\varepsilon}\right)} & {=\operatorname{div}\phi} & {} & {\text { in } \Omega} \\ {w_{\varepsilon}} & {=0} & {} & {\text { on } \partial \Omega},\end{array}\right.
\end{equation} where $\tilde{f}=(\tilde{f}_i^\alpha)$ and \begin{equation}\tilde{f_i^\alpha}=H_{1,i}^\alpha+H_{21,i}^\alpha+H_{22,i}^\alpha+H_{23,i}^\alpha+H_{3,i}^\alpha+H_{4,i}^\alpha\end{equation} with $H_{j2,i}^\alpha$ (j=1,2,3) defined in $(3.14)$, $(3.15)$ and $(3.16)$, respectively.\end{lemma}
\begin{proof}According to the first line of the equation $(3.4)$,
\begin{equation}
\left[\mathcal{L}_{\varepsilon}\left(w_{\varepsilon}\right)\right]^{\alpha}=-\nabla_{i} f_{i}^{\alpha}-\nabla_{\alpha}\left(p_{\varepsilon}-p_{0}\right), \quad \text { in } \Omega.
\end{equation}

To obtain the first line of $(3.10)$, we need only to check the term $H_{2,i}^\alpha$ in $f_i^\alpha$.
We firstly observe that
\begin{equation}H_{2,i}^\alpha=(I_{1,ij}^{\alpha\beta}(x/\varepsilon,x/\varepsilon^2)+I_{2,ij}^{\alpha\beta}(x/\varepsilon)+I_{3,ij}^{\alpha\beta}(x/\varepsilon,x/\varepsilon^2))\psi_{2\varepsilon}S_\varepsilon(\partial_{j}u_0^\beta)\end{equation} with $I_{1,ij}^{\alpha\beta}$, $I_{2,ij}^{\alpha\beta}$ and $I_{3,ij}^{\alpha\beta}$ defined in $(2.15)$, $(2.20)$ and $(2.23)$, respectively.

According to $(\ref{1.6})$, we have $\iint_{Y\times Z}I_{1,ij}^{\alpha\beta}(y,z)+I_{2,ij}^{\alpha\beta}(y)+I_{3,ij}(y,z)dydz=0$, then $\iint_{Y\times Z}I_{2,ij}^{\alpha\beta}(y)dydz=0$ due to Lemma 2.6 (i) and Lemma 2.8 (i). Therefore, we have
$\int_{Y}I_{2,ij}^{\alpha\beta}(y)dy=0$ which satisfies the assumption of Lemma 2.7.

In view of $(2.16)$ and recalling $y=x/\varepsilon,z=x/\varepsilon^2$, we have
\begin{equation}\begin{aligned}
&\partial_{x_i}\left[I_{1,ij}^{\alpha\beta}(y,z)\psi_{2\varepsilon}S_\varepsilon(\partial_{j}u_0^\beta)\right]\\
=&\partial_{x_i}\left[\left(\partial_{z_k}E^{\alpha\beta}_{1,kij}(y,z)+\partial_{z_\alpha}q_{1,ik}^\beta(y,z)\right)\psi_{2\varepsilon}S_\varepsilon(\partial_{j}u_0^\beta)\right]\\
=&\partial_{x_i}\left[\left(\varepsilon^2\partial_{x_k}E^{\alpha\beta}_{1,kij}(y,z)-\varepsilon\partial_{y_k}E^{\alpha\beta}_{1,kij}(y,z)+\varepsilon^2\partial_{x_\alpha}q_{1,ik}^\beta(y,z)-\varepsilon\partial_{y_\alpha}q_{1,ik}^\beta(y,z)\right)\psi_{2\varepsilon}S_\varepsilon(\partial_{j}u_0^\beta)\right]\\
=&\partial_{x_i}\left[\left(-\varepsilon\partial_{y_k}E^{\alpha\beta}_{1,kij}(y,z)-\varepsilon\partial_{y_\alpha}q_{1,ik}^\beta(y,z)\right)\psi_{2\varepsilon}S_\varepsilon(\partial_{j}u_0^\beta)\right]\\
&+\partial_{x_i}\left[\partial_{x_k}\left(\varepsilon^2E^{\alpha\beta}_{1,kij}(y,z)\psi_{2\varepsilon}S_\varepsilon(\partial_{j}u_0^\beta)\right)-\varepsilon^2E^{\alpha\beta}_{1,kij}(y,z)\partial_{x_k}\left(\psi_{2\varepsilon}S_\varepsilon(\partial_{j}u_0^\beta)\right)\right]\\
&+\partial_{x_i}\left[\varepsilon^2\partial_{x_\alpha}\left(q_{1,ik}^\beta(y,z)\psi_{2\varepsilon}S_\varepsilon(\partial_{j}u_0^\beta)\right)-\varepsilon^2q_{1,ik}^\beta(y,z)\partial_{x_\alpha}\left(\psi_{2\varepsilon}S_\varepsilon(\partial_{j}u_0^\beta)\right)\right]\\
=&\partial_{x_i}\left[\left(-\varepsilon\partial_{y_k}E^{\alpha\beta}_{1,kij}(y,z)-\varepsilon\partial_{y_\alpha}q_{1,ik}^\beta(y,z)\right)\psi_{2\varepsilon}S_\varepsilon(\partial_{j}u_0^\beta)
-\varepsilon^2E^{\alpha\beta}_{1,kij}(y,z)\partial_{x_k}\left(\psi_{2\varepsilon}S_\varepsilon(\partial_{j}u_0^\beta)\right)\right.\\
&\ \ \ \ \ -\left.\varepsilon^2q_{1,ik}^\beta(y,z)\partial_{x_\alpha}\left(\psi_{2\varepsilon}S_\varepsilon(\partial_{j}u_0^\beta)\right)\right]
+\partial_{x_\alpha}\left[\varepsilon^2\partial_{x_i}\left(q_{1,ik}^\beta(y,z)\psi_{2\varepsilon}S_\varepsilon(\partial_{j}u_0^\beta)\right)\right]\\
=&:\partial_{x_i}H_{21,i}^\alpha+\partial_{x_\alpha}\left[\varepsilon^2\partial_{x_i}\left(q_{1,ik}^\beta(y,z)\psi_{2\varepsilon}S_\varepsilon(\partial_{j}u_0^\beta)\right)\right]\\
\end{aligned}\end{equation} Similarly, we have
\begin{equation}\begin{aligned}
&\partial_{x_i}\left[I_{1,ij}^{\alpha\beta}(y)\psi_{2\varepsilon}S_\varepsilon(\partial_{j}u_0^\beta)\right]\\
=&\partial_{x_i}\left[-\varepsilon E^{\alpha\beta}_{2,kij}(y)\partial_{x_k}\left(\psi_{2\varepsilon}S_\varepsilon(\partial_{j}u_0^\beta)\right)-\varepsilon q_{2,ik}^\beta(y)\partial_{x_\alpha}\left(\psi_{2\varepsilon}S_\varepsilon(\partial_{j}u_0^\beta)\right)\right]\\
&+\partial_{x_\alpha}\left[\varepsilon \partial_{x_i}\left(q_{2,ik}^\beta(y)\psi_{2\varepsilon}S_\varepsilon(\partial_{j}u_0^\beta)\right)\right]\\
=&:\partial_{x_i}H_{22,i}^\alpha+\partial_{x_\alpha}\left[\varepsilon \partial_{x_i}\left(q_{2,ik}^\beta(y)\psi_{2\varepsilon}S_\varepsilon(\partial_{j}u_0^\beta)\right)\right],
\end{aligned}\end{equation}
and \begin{equation}\begin{aligned}
&\partial_{x_i}\left[I_{3,ij}^{\alpha\beta}(y,z)\psi_{2\varepsilon}S_\varepsilon(\partial_{j}u_0^\beta)\right]\\
=&\partial_{x_i}\left[\left(-\varepsilon\partial_{y_k}E^{\alpha\beta}_{3,kij}(y,z)-\varepsilon\partial_{y_\alpha}q_{3,ik}^\beta(y,z)\right)\psi_{2\varepsilon}S_\varepsilon(\partial_{j}u_0^\beta)
-\varepsilon^2E^{\alpha\beta}_{3,kij}(y,z)\partial_{x_k}\left(\psi_{2\varepsilon}S_\varepsilon(\partial_{j}u_0^\beta)\right)\right.\\
&\ \ \ \ \ -\left.\varepsilon^2q_{3,ik}^\beta(y,z)\partial_{x_\alpha}\left(\psi_{2\varepsilon}S_\varepsilon(\partial_{j}u_0^\beta)\right)\right]
+\partial_{x_\alpha}\left[\varepsilon^2\partial_{x_i}\left(q_{3,ik}^\beta(y,z)\psi_{2\varepsilon}S_\varepsilon(\partial_{j}u_0^\beta)\right)\right]\\
=&:\partial_{x_i}H_{23,i}^\alpha+\partial_{x_\alpha}\left[\varepsilon^2\partial_{x_i}\left(q_{3,ik}^\beta(y,z)\psi_{2\varepsilon}S_\varepsilon(\partial_{j}u_0^\beta)\right)\right].
\end{aligned}\end{equation} Consequently, combining $(3.13)-(3.16)$ gives the desired equation $(3.10)$.
\end{proof}

In order to obtain the error estimates, we firstly give the estimate of $||H_{21,i}^\alpha+H_{22,i}^\alpha+H_{23,i}^\alpha||_{L^2(\Omega)}$ by using the method of Fourier transform to separate the different scales of $x$.
\begin{lemma} $||H_{21,i}^\alpha+H_{22,i}^\alpha+H_{23,i}^\alpha||_{L^2(\Omega)}\leq C\varepsilon ||\nabla u_0||_{L^2(\Omega)}.$ \end{lemma}
\begin{proof}We note that the estimate of $H_{23,i}^\alpha$ is the most difficult to handle, therefore we need only to estimate $||H_{23,i}^\alpha||_{L^2(\Omega)}$, since the others are even easier and totally similarly to $H_{23,i}^\alpha$. Recall that
\begin{equation}H_{23,i}^\alpha=:T_1+T_2+T_3+T_4\end{equation}
with
\begin{equation}\begin{aligned}
&T_1=-\varepsilon\partial_{y_k}E^{\alpha\beta}_{3,kij}(y,z)\psi_{2\varepsilon}S_\varepsilon(\partial_{j}u_0^\beta),\\
&T_2=-\varepsilon\partial_{y_\alpha}q_{3,ik}^\beta(y,z)\psi_{2\varepsilon}S_\varepsilon(\partial_{j}u_0^\beta),\\
&T_3=-\varepsilon^2E^{\alpha\beta}_{3,kij}(y,z)\partial_{x_k}\left(\psi_{2\varepsilon}S_\varepsilon(\partial_{j}u_0^\beta)\right),\\
&T_4=-\varepsilon^2q_{3,ik}^\beta(y,z)\partial_{x_\alpha}\left(\psi_{2\varepsilon}S_\varepsilon(\partial_{j}u_0^\beta)\right).
\end{aligned}\end{equation} We note that $I_{3,ij}^{\alpha\beta}(y,z)$ is Y-Z periodic due to $A(y,z)$ is Y-Z periodic. Recall that we assume that $Y=Z=(0,1)^n.$ Taking the Fourier transform of $I^{\alpha\beta}_{3,ij}(y,z)$ with respect to $z$ gives that
\begin{equation}\label{3.9}I_{3,ij}^{\alpha\beta}(y,z)=\sum_{k\in\mathbb{Z}^n}\widehat{I}_{3,kij}^{\alpha\beta}(y)e^{2\pi\sqrt{-1} kz},\end{equation}
where $\widehat{I}_{3,kij}^{\alpha\beta}(y)$ is given by
\begin{equation}\widehat{I}_{3,kij}^{\alpha\beta}(y)=\int_{(0,1)^n}I_{3,ij}^{\alpha\beta}(y,z)e^{-2\pi\sqrt{-1} kz}dz.\end{equation}
Clearly, as the notations in Lemma 2.8, we have
$$\begin{aligned}&q_{3,ij}^\beta(y,z)=\frac{\sqrt{-1}}{2\pi}\sum_{k\in\mathbb{Z}^n}\frac{k_\alpha}{|k|^{2}}\widehat{I}^{\alpha\beta}_{3,kij}(y)e^{2\pi\sqrt{-1} kz},\\ &f^{\alpha\beta}_{3,ij}(y,z)=-\frac{1}{4\pi^2}\sum_{k\in\mathbb{Z}^n}|k|^{-2}\widehat{I}_{3,kij}^{\alpha\beta}(y)e^{2\pi\sqrt{-1} kz}+\frac{1}{4\pi^2}\sum_{k\in\mathbb{Z}^n}\frac{k_\alpha k_\eta}{|k|^4}\widehat{I}^{\eta\beta}_{3,kij}(y)e^{2\pi\sqrt{-1}},\end{aligned}$$
and
\begin{equation}\label{3.11}\begin{aligned}
E_{3,hij}^{\alpha\beta}(y,z)&=\partial_{z_h}f^{\alpha\beta}_{3,ij}(y,z)-\partial_{z_i}f^{\alpha\beta}_{3,hj}(y,z)\\
&=\frac{\sqrt{-1}}{2\pi}\sum_{k\in\mathbb{Z}^n}k_i|k|^{-2}\widehat{I}^{\alpha\beta}_{3,khj}(y)e^{2\pi\sqrt{-1} kz}-
\frac{\sqrt{-1}}{2\pi}\sum_{k\in\mathbb{Z}^n}k_h|k|^{-2}\widehat{I}^{\alpha\beta}_{3,kij}(y)e^{2\pi\sqrt{-1} kz}\\
&\ \ \ -\frac{\sqrt{-1}}{2\pi}\sum_{k\in\mathbb{Z}^n}\frac{k_ik_\alpha k_\eta}{|k|^4}\widehat{I}^{\eta\beta}_{3,khj}(y)e^{2\pi\sqrt{-1} kz}+
\frac{\sqrt{-1}}{2\pi}\sum_{k\in\mathbb{Z}^n}\frac{k_hk_\alpha k_\eta}{|k|^4}\widehat{I}^{\eta\beta}_{3,kij}(y)e^{2\pi\sqrt{-1} kz}.\\
\end{aligned}\end{equation}

Then according to Lemma 2.10 (i), \begin{equation}\begin{aligned}||T_1||^2_{L^2(\Omega)}\leq C \varepsilon^2 \int_\Omega\left(\sum_{k\in\mathbb{Z}^n}|k|^{-1}|\nabla_y\widehat{I}^{\alpha\beta}_{3,kij}(x/\varepsilon)|\right)^2
\left|\psi_{2\varepsilon}S_\varepsilon(\nabla u_0)\right|^2dx\\
\leq C \varepsilon^4 \int_Y\left(\sum_{k\in\mathbb{Z}^n}|k|^{-1}|\nabla_y\widehat{I}^{\alpha\beta}_{3,kij}(y)|\right)^2dy\cdot||\nabla u_0||_{L^2(\Omega)}^2.
\end{aligned}\end{equation}
H\"{o}lder's inequality and Plancherel's Indetity  give that
\begin{equation}\begin{aligned}\int_{Y}\left(\sum_{k\in\mathbb{Z}^n}|k|^{-1}|\nabla_y\widehat{I}^{\alpha\beta}_{3,kij}(y)|\right)^2dy&\leq
\int_{Y}\sum_{k\in\mathbb{Z}^n}\left|\nabla_y\widehat{I}^{\alpha\beta}_{3,kij}(y)\right|^2dy\cdot \sum_{k\in\mathbb{Z}^n}|k|^{-2}\\
&\leq C\iint_{Y\times Z}\left|\nabla_yI^{\alpha\beta}_{3,ij}(y,z)\right|^2dzdy,
\end{aligned}\end{equation} in view of $(2.23)$,
\begin{equation}\label{2.17}\begin{aligned}
I_{3,ij}^{\alpha\beta}(y,z)=&a_{ik}^{\alpha\gamma}(y,z)\partial_{y_k}\chi_j^{\gamma\beta}(y)-a_{ik}^{\alpha\eta}(y,z)\partial_{z_k}\chi_l^{\eta\gamma}(y,z)\partial_{y_l}\chi_j^{\gamma\beta}(y)\\
&-\fint_Z\left(a_{ik}^{\alpha\gamma}(y,z)\partial_{y_k}\chi_j^{\gamma\beta}(y)-a_{ik}^{\alpha\gamma}(y,z)\partial_{z_k}\chi_l^{\gamma\eta}(y,z)\partial_{y_l}\chi_j^{\eta\beta}(y)\right)dz\\
=:&I_{31,ij}^{\alpha\beta}+I_{32,ij}^{\alpha\beta}+I_{33,ij}^{\alpha\beta},
\end{aligned}\end{equation}

then according  to $(\ref{1.3})$ and $(\ref{2.2})$,
\begin{equation}\iint_{Y\times Z}\left|\nabla_yI^{\alpha\beta}_{31,ij}(y,z)\right|^2dzdy\leq C,\end{equation}

and
\begin{equation}\label{3.20}\begin{aligned}
||\nabla_yI^{\alpha\beta}_{32,ij}(y,z)||_{L^2(Y\times Z)}\leq& ||\nabla_y A(y,z)||_{L^\infty(Y\times Z)}||\nabla_y \chi (y)||_{L^\infty(Y)}||\nabla_z \chi (y,z)||_{L^2(Y\times Z)}\\
&+||A(y,z)||_{L^\infty(Y\times Z)}||\nabla_y \chi(y)||_{L^\infty(Y)}||\nabla_y\nabla_z \chi(y,z)||_{L^2(Y\times Z)}\\
&+||A(y,z)||_{L^\infty(Y\times Z)}||\nabla^2_y \chi(y)||_{L^{(4+2\tau)/\tau}(Y)}||\nabla_z \chi(y,z)||_{L^{2+\tau}(Y\times Z)}\\
\leq& C,
\end{aligned}\end{equation} where we have used $(\ref{1.3})$, $(\ref{2.1})$, $(\ref{2.2})$, $(\ref{2.7})$ and the Sobolev embedding inequality in the above inequality. Therefore, combining $(3.22)-(3.26)$ gives that
\begin{equation}
||T_1||_{L^2(\Omega)}\leq C\varepsilon^2 ||\nabla u_0||_{L^2(\Omega)}.
\end{equation} Note that $q_{3,ij}^\beta$ has the similar form as $E_{3,hij}^{\alpha\beta}$, then we also have
\begin{equation}
||T_2||_{L^2(\Omega)}\leq C\varepsilon^2 ||\nabla u_0||_{L^2(\Omega)}.
\end{equation}
And similarly, according to the second line of $(2.27)$, we have $$||T_3||_{L^2(\Omega)}+||T_4||_{L^2(\Omega)}\leq C\varepsilon ||\nabla u_0||_{L^2(\Omega)}.$$
Thus we complete this proof.\end{proof}

\begin{thm}Under the assumptions in Lemma 3.2, then we have the following estimates
\begin{equation}||w_\varepsilon||_{H_0^1(\Omega)}+||z_\varepsilon-\int_\Omega z_\varepsilon||_{L^2(\Omega)}\leq C\varepsilon||\nabla^2 u_0||_{L^2(\Omega)}
+C||\nabla u_0||_{L^2(\Omega\setminus\Sigma_{5\varepsilon})}+C\varepsilon||\nabla u_0||_{L^2(\Omega)}.\end{equation}\end{thm}
\begin{proof}According to $(3.10)$ and Lemma 2.2, \begin{equation}||w_\varepsilon||_{H_0^1(\Omega)}+||z_\varepsilon-\int_\Omega z_\varepsilon||_{L^2(\Omega)}\leq C (||\tilde{f}||_{L^2(\Omega)}+||\operatorname{div}\phi||_{L^2(\Omega)}).\end{equation}\end{proof}
In view of $(3.8)$,
\begin{equation}\begin{aligned}||H_{1,i}||_{L^2(\Omega)}\leq C||\nabla u_0-S_\varepsilon\left(\nabla u_0\right)||_{L^2(\Omega)}+C|| S_\varepsilon(\nabla u_0)-\psi_{2\varepsilon}S_\varepsilon(\nabla u_0)||_{L^2(\Omega)}\\
\leq C\varepsilon||\nabla^2 u_0||_{L^2(\Omega)}
+C||\nabla u_0||_{L^2(\Omega\setminus\Sigma_{5\varepsilon})},
\end{aligned}\end{equation}

Note that $\chi(y,z)$ is also Y-Z periodic, then imitating the proof of the estimate of $T_1$ in Lemma 3.3, we can obtain
\begin{equation}||H_{3,i}||_{L^2(\Omega)}+||H_{4,i}||_{L^2(\Omega)}\leq C\varepsilon||\nabla^2 u_0||_{L^2(\Omega)}
+C||\nabla u_0||_{L^2(\Omega\setminus\Sigma_{5\varepsilon})}+C\varepsilon||\nabla u_0||_{L^2(\Omega)}.\end{equation}

Therefore, \begin{equation}||\tilde{f}||_{L^2(\Omega)}\leq C\varepsilon||\nabla^2 u_0||_{L^2(\Omega)}
+C||\nabla u_0||_{L^2(\Omega\setminus\Sigma_{5\varepsilon})}+C\varepsilon||\nabla u_0||_{L^2(\Omega)}.\end{equation}

Next, we need to estimate $||\operatorname{div}\phi||_{L^2(\Omega)}$. In view of the definition $(3.6)$ of $\phi$, the first term in $\operatorname{div}\phi$ is easy to
estimate after noting $\operatorname{div}\chi_k^\alpha(x)=0$. The second term and the third term have the similar estimates after noting that $||\nabla_y \chi_k(y)||_{L^\infty}\leq C$, therefore, we just give the
estimate of the second term by using the Fourier transform methods to separate the different scales of $x$. Due to $A(y,z)$ is Y-Z periodic, then $\chi_j^\alpha(y,z)$ is also Y-Z periodic. Taking the Fourier transform with respect to $y$ of $\chi_j^\alpha(y,z)$ leads to
\begin{equation*}\chi_j^\alpha(y,z)=\sum_{k\in\mathbb{Z}^n}\widehat{\chi}_{kj}^\alpha(z)e^{2\pi\sqrt{-1} ky},\end{equation*}
where $\widehat{\chi}^\alpha_{kj}(z)$ is given by
\begin{equation*}\widehat{\chi}_{kj}^\alpha(z)=\int_{(0,1)^n}\chi_j^\alpha(y,z)e^{-2\pi\sqrt{-1} ky}dy.\end{equation*}
And $\operatorname{div}_z\chi_j^\alpha(y,z)=0$ yields $$\sum_{k\in\mathbb{Z}^n}\partial_{z_\beta}\widehat{\chi}_{kj}^{\beta\alpha}(z)e^{2\pi\sqrt{-1} ky}=0.$$
In view of $(3.6)$, the second term of $\operatorname{div}\phi$ is given by \begin{equation}\begin{aligned}&\varepsilon\operatorname{div}_y\chi_j^{\alpha}(x/\varepsilon,x/\varepsilon^2)\psi_{2\varepsilon}S_\varepsilon\left(\partial_{x_j}u_0^\alpha\right)
+\operatorname{div}_z\chi_j^{\alpha}(x/\varepsilon,x/\varepsilon^2)\psi_{2\varepsilon}S_\varepsilon\left(\partial_{x_j}u_0^\alpha\right)\\
&\ \ +\varepsilon^2\chi_j^{\alpha}(x/\varepsilon,x/\varepsilon^2)\operatorname{div}(\psi_{2\varepsilon}S_\varepsilon\left(\partial_{x_j}u_0^\alpha\right))\\
=&:J_1+J_2+J_3,
\end{aligned}\end{equation}

Similar to the proof of $T_1$ in Lemma 3.3, $J_1+J_3$ is easy to estimate. To estimate $J_2$ more accurately, collect a family of small cubes by $Z_{\varepsilon^2}^i=\varepsilon^2(i+Z)$ for
$i\in \mathbb{Z}^n$ with an index set $I_{\varepsilon^2}$, such that $\Sigma_\varepsilon\subset\cup_{i\in I_{\varepsilon^2}}\subset \Omega$,  and
$Z_{\varepsilon^2}^i\cap Z_{\varepsilon^2}^j={\o}$ if $i\neq j$. Therefore,
\begin{equation}\begin{aligned}
&||\operatorname{div}_z\chi_j^{\alpha}(y,z)\psi_{2\varepsilon}S_\varepsilon\left(\partial_{x_j}u_0^\alpha\right)||^2_{L^2(\Omega)}\\
=&4\pi^2\int_\Omega\left|\sum_{k\in\mathbb{Z}^n}\partial_{z_\beta}\widehat{\chi}_{kj}^{\beta\alpha}(x/\varepsilon^2)e^{2\pi\sqrt{-1} kx/\varepsilon}\psi_{2\varepsilon}
S_\varepsilon(\partial_j u_0^\alpha)\right|^2dx\\
\leq& C\sum_{i\in I_{\varepsilon^2}}\int_{Z_{\varepsilon^2}^i}\left|\sum_{k\in\mathbb{Z}^n}\partial_{z_\beta}\widehat{\chi}_{kj}^{\beta\alpha}(x/\varepsilon^2)e^{2\pi\sqrt{-1} kx/\varepsilon}\psi_{2\varepsilon}
S_\varepsilon(\partial_j u_0^\alpha)\right|^2dx\\
\leq& C\sum_{i\in I_{\varepsilon^2}}\int_{Z_{\varepsilon^2}^i}\left|\sum_{k\in\mathbb{Z}^n}\partial_{z_\beta}\widehat{\chi}_{kj}^{\beta\alpha}(x/\varepsilon^2)\left(e^{2\pi\sqrt{-1} kx/\varepsilon}-e^{2\pi\sqrt{-1} kz^i/\varepsilon}\right)
\psi_{2\varepsilon}S_\varepsilon(\partial_j u_0^\alpha)\right|^2dx\\
\leq& C\varepsilon^2\sum_{i\in I_{\varepsilon^2}}\int_{Z_{\varepsilon^2}^i}\left|\sum_{k\in\mathbb{Z}^n}|\partial_{z_\beta}\widehat{\chi}_{kj}^{\beta\alpha}(x/\varepsilon^2)|
\cdot\psi_{2\varepsilon}\cdot|S_\varepsilon(\partial_j u_0^\alpha)|\right|^2dx\\
\leq& C\varepsilon^2\int_{\Omega}\left|\sum_{k\in\mathbb{Z}^n}|\partial_{z_\beta}\widehat{\chi}_{kj}^{\beta\alpha}(x/\varepsilon^2)|
\cdot\psi_{2\varepsilon}\cdot|S_\varepsilon(\partial_j u_0^\alpha)|\right|^2dx,
\end{aligned}\end{equation}where $z^i$ is the center of $Z_{\varepsilon^2}^i$. H\"{o}lder's inequality and Plancherel's Indetity  give that
\begin{equation}\begin{aligned}\int_{Z}\left(\sum_{k\in\mathbb{Z}^n}|\nabla_{z}\widehat{\chi}_{kj}^{\alpha}(z)|\right)^2dz&\leq
\int_{Y}\sum_{k\in\mathbb{Z}^n}|k|^2\left|\nabla_{z}\widehat{\chi}_{kj}^{\alpha}(z)\right|^2dz\cdot \sum_{k\in\mathbb{Z}^n}|k|^{-2}\\
&\leq C\iint_{Y\times Z}\left|\nabla_y\nabla_z\chi_j^\alpha(y,z)\right|^2dzdy\leq C,
\end{aligned}\end{equation} where we have used $(2.6)$ in the above inequality.
Consequently, according to $(2.28)$, $(3.35)$ and $(3.36)$, $$||\operatorname{div}_z\chi_j^{\alpha}(y,z)\psi_{2\varepsilon}S_\varepsilon\left(\partial_{x_j}u_0^\alpha\right)||_{L^2(\Omega)}\leq C \varepsilon^3||\nabla u_0||_{L^2(\Omega)}.$$

Generally, we can estimate $||\nabla \phi||_{L^2(\Omega)}$. The difference from the proof of $H_{2,i}$ is that when estimating the term $||\nabla_y\chi(x/\varepsilon,x/\varepsilon^2)\psi_{2\varepsilon}S_\varepsilon(\nabla u_0)||_{L^2(\Omega)}$, we need to take the Fourier transform of $\nabla_y\chi(y,z)$ with respect to $z$ (then $(2.28)$ will be applicable);  and when estimating the term
$||\nabla_z\chi(x/\varepsilon,x/\varepsilon^2)\psi_{2\varepsilon}S_\varepsilon(\nabla u_0)||_{L^2(\Omega)}$, we need to take the Fourier transform of
$\nabla_z\chi(y,z)$ with respect to $y$ (then $(2.27)$ will be applicable). As a result, we will have
\begin{equation}||\nabla \phi||_{L^2(\Omega)}\leq C\varepsilon||\nabla^2 u_0||_{L^2(\Omega)}
+C||\nabla u_0||_{L^2(\Omega\setminus\Sigma_{5\varepsilon})}+C\varepsilon||\nabla u_0||_{L^2(\Omega)}.\end{equation}
Consequently, the desired estimate $(3.29)$ follows from $(3.33)$ and $(3.37)$.
\begin{rmk}In order to obtain better estimates, if $w_\varepsilon$ has the form
\begin{equation}\label{3.2}\begin{aligned}
w_\varepsilon^\beta(x)=&u_\varepsilon^\beta(x)-u_0^\beta(x)+\varepsilon \chi_j^{\beta\gamma}(x/\varepsilon)\psi_{2\varepsilon^\lambda}S_{\varepsilon^\lambda}\left(\partial_{x_j}u_0^\gamma\right)\\
&+\varepsilon^2\chi_j^{\beta\alpha}(x/\varepsilon,x/\varepsilon^2)\left[\psi_{2\varepsilon^\lambda}S_{\varepsilon^\lambda}\left(\partial_{x_j}u_0^\alpha\right)
-\partial_{y_j}\chi_{k}^{\alpha\gamma}(x/\varepsilon)\psi_{2\varepsilon^\lambda}S_{\varepsilon^\lambda}\left(\partial_{x_k}u_0^\gamma\right)\right],
\end{aligned}\end{equation}
where $0<\lambda$ is a constant which is to be chosen. In view of Remark 2.8, we need to assume $\lambda\leq 2$. (Actually, in view of the term
$a_{ih}^{\alpha\beta}\partial_{y_h}\chi_j^{\beta\lambda}(y,z)\partial_{y_j}\chi_{k}^{\gamma\eta}(z)\psi_{2\varepsilon^\lambda}S_{\varepsilon^\lambda}\left(\partial_{x_k}u_0^\eta\right)$ in $H_{3,i}^\alpha$ defined in $(\ref{3.5})$, we need $\lambda \leq 2$.  However, if $\lambda>1$, we need more regularity assumptions on $\chi_k^\beta(y)$ and $\chi_k^\beta(y,z)$). Consequently, careful computation shows that $\lambda=1$ is the best choice, which may declare that the scale of $\varepsilon$ dominates any other scales. The same
result holds for $w_\varepsilon$ of the form
\begin{equation}\label{3.2}\begin{aligned}
w_\varepsilon^\beta(x)=&u_\varepsilon^\beta(x)-u_0^\beta(x)+\varepsilon \chi_j^{\beta\gamma}(x/\varepsilon)\psi_{2\varepsilon^\mu}S_{\varepsilon^\lambda}\left(\partial_{x_j}u_0^\gamma\right)\\
&+\varepsilon^2\chi_j^{\beta\alpha}(x/\varepsilon,x/\varepsilon^2)\left[\psi_{2\varepsilon^\mu}S_{\varepsilon^\lambda}\left(\partial_{x_j}u_0^\alpha\right)
-\partial_{y_j}\chi_{k}^{\alpha\gamma}(x/\varepsilon)\psi_{2\varepsilon^\mu}S_{\varepsilon^\lambda}\left(\partial_{x_k}u_0^\gamma\right)\right],
\end{aligned}\end{equation} where $0<\mu<\lambda\leq 2$.
\end{rmk}
\section{Proof of Theorem 1.2, convergence rates for the pressure term}
To obtain the convergence rates for the pressure term, we introduce the following lemma whose proof may be founded in \cite{Suslina2012Homogenization}.
\begin{lemma}Let $\Omega\subset\mathbb{R}^n$ be a bounded $C^1$ domain. Then, for any function $u\in H^1(\Omega)$,
\begin{equation}\int_{\Omega\setminus\Sigma_{\varepsilon}}|u|^2\leq C\varepsilon||u||_{H^1(\Omega)}||u||_{L^2(\Omega)},\end{equation}
where $C$ depends only on $\Omega$.\end{lemma} In view of the definition $(3.10)$ of $z_\varepsilon$,
\begin{equation}\begin{aligned}
z_\varepsilon&=p_\varepsilon-p_0+\varepsilon\partial_{y_i}q_{1,ik}^\beta(y,z)\psi_{2\varepsilon}S_\varepsilon(\partial_{j}u_0^\beta)
+\partial_{z_i}q_{1,ik}^\beta(y,z)\psi_{2\varepsilon}S_\varepsilon(\partial_{j}u_0^\beta)
+\varepsilon\partial_{y_i}q_{2,ik}^\beta(y)\psi_{2\varepsilon}S_\varepsilon(\partial_{k}u_0^\beta)\\
&\ \ +\varepsilon^2q_{1,ik}^\beta(y,z)\partial_{x_i}(\psi_{2\varepsilon}S_\varepsilon(\partial_{k}u_0^\beta))
+q_{2,ik}^\beta(y)\partial_{x_i}(\psi_{2\varepsilon}S_\varepsilon(\partial_{k}u_0^\beta))
+\varepsilon\partial_{y_i}q_{3,ik}^\beta(y,z)\psi_{2\varepsilon}S_\varepsilon(\partial_{k}u_0^\beta)\\
&\ \ +\partial_{z_i}q_{3,ik}^\beta(y,z)\psi_{2\varepsilon}S_\varepsilon(\partial_{k}u_0^\beta)
+\varepsilon^2q_{3,ik}^\beta(y,z)\partial_{x_i}(\psi_{2\varepsilon}S_\varepsilon(\partial_{k}u_0^\beta))\\
&=p_\varepsilon-p_0+\pi_k^\beta(y,z)\psi_{2\varepsilon}S_\varepsilon(\partial_{k}u_0^\beta)+\pi_k^\beta(y)\psi_{2\varepsilon}S_\varepsilon(\partial_{k}u_0^\beta)
-\pi_j^\gamma(y,z)\partial_{y_j}\chi_k^{\gamma\beta}(y)\psi_{2\varepsilon}S_\varepsilon(\partial_{k}u_0^\beta)\\
&\ \ +\varepsilon\partial_{y_i}q_{1,ik}^\beta(y,z)\psi_{2\varepsilon}S_\varepsilon(\partial_{k}u_0^\beta)
+\varepsilon q_{2,ik}^\beta(y)\partial_{x_i}(\psi_{2\varepsilon}S_\varepsilon(\partial_{k}u_0^\beta))
+\varepsilon\partial_{y_i}q_{3,ik}^\beta(y,z)\psi_{2\varepsilon}S_\varepsilon(\partial_{k}u_0^\beta)\\
&\ \ +\varepsilon^2q_{3,ik}^\beta(y,z)\partial_{x_i}(\psi_{2\varepsilon}S_\varepsilon(\partial_{k}u_0^\beta))
+\varepsilon^2q_{1,ik}^\beta(y,z)\partial_{x_i}(\psi_{2\varepsilon}S_\varepsilon(\partial_{k}u_0^\beta))\\
&=:p_\varepsilon-p_0+\pi_k^\beta(y,z)\psi_{2\varepsilon}S_\varepsilon(\partial_{k}u_0^\beta)+\pi_k^\beta(y)\psi_{2\varepsilon}S_\varepsilon(\partial_{k}u_0^\beta)
-\pi_j^\gamma(y,z)\partial_{y_j}\chi_k^{\gamma\beta}(y)\psi_{2\varepsilon}S_\varepsilon(\partial_{k}u_0^\beta)\\
&\ \ +\sum_{h=1}^{5}T_h,\\
&\text{with } y=x/\varepsilon,\ z=x/\varepsilon^2.
\end{aligned}\end{equation} where we have used $(2.16)$. $(2.21)$ and $(2.24)$ in the second equation. By using the Fourier transform methods (see the proof of $T_2$ in Lemma 3.3), we can obtain
\begin{equation}\begin{aligned}
\sum_{h_1}^5||T_h||_{L^2(\Omega)}&\leq C\varepsilon||\nabla^2 u_0||_{L^2(\Omega)}
+C||\nabla u_0||_{L^2(\Omega\setminus\Sigma_{5\varepsilon})}+C\varepsilon||\nabla u_0||_{L^2(\Omega)}\\
&\leq C\varepsilon^{1/2}||u_0||_{H^2(\Omega)}.
\end{aligned}\end{equation}Consequently, in view of $(3.30)$, we have
\begin{equation}||p_\varepsilon-p_0+\tilde{\pi}-\int_\Omega\tilde{\pi}||\leq C\varepsilon^{1/2}||u_0||_{H^2(\Omega)},\end{equation}
where $\tilde{\pi}=\pi_k^\beta(y,z)\psi_{2\varepsilon}S_\varepsilon(\partial_{k}u_0^\beta)+\pi_k^\beta(y)\psi_{2\varepsilon}S_\varepsilon(\partial_{k}u_0^\beta)
-\pi_j^\gamma(y,z)\partial_{y_j}\chi_k^{\gamma\beta}(y)\psi_{2\varepsilon}S_\varepsilon(\partial_{k}u_0^\beta)$ with $y=x/\varepsilon$ and $z=x/\varepsilon^2$, thus we complete the proof of Theorem 1.2.
\section{Proof of Theorem 1.1, convergence rates for the velocity term}

In this section, we study the convergence rates in $L^2$ and give the proof of Theorem 1.1. Actually, the proof of Theorem 1.1 is achieved by duality.
So we need the consider the adjoint problems: For any $G\in L^2(\Omega;\mathbb{R}^n)$, there exist $(v_\varepsilon,\theta_\varepsilon)$, $(v_0,\theta_0)\in H^1_0(\Omega;\mathbb{R}^n)\times L^2(\Omega)$ respectively solving
\begin{equation}
\left\{
\begin{aligned}
\mathcal{L}_{\varepsilon}^* v_{\varepsilon} +\nabla \theta_\varepsilon&=G  \quad\text { in } \Omega, \\
\operatorname{div}v_\varepsilon&=0 \quad\text { in } \Omega,\\
v_{\varepsilon}&=0  \quad\text { on } \partial \Omega,
\end{aligned}\right.
\end{equation} and
\begin{equation}
\left\{
\begin{aligned}
\mathcal{L}_{0} v_{0} +\nabla \theta_0&=G  \quad\text { in } \Omega, \\
\operatorname{div}v_0&=0 \quad\text { in } \Omega,\\
v_{0}&=0  \quad\text { on } \partial \Omega
\end{aligned}\right.
\end{equation}
with $$\int_\Omega \theta_\varepsilon=\int_\Omega \theta_0=0.$$

 Here we have used the notation: $\mathcal{L}_\varepsilon^*=-\operatorname{div}(A^*(x/\varepsilon,x/\varepsilon^2)\nabla)$ and $\mathcal{L}_0^*=-\operatorname{div}(\widehat{A}^*\nabla)$.
Moreover, we denote
\begin{equation}\begin{aligned}
\tilde{w}_\varepsilon^\beta(x)=&v_\varepsilon^\beta(x)-v_0^\beta(x)+\varepsilon \chi_j^{*,\beta\gamma}(x/\varepsilon)\psi_{10\varepsilon}S_\varepsilon\left(\partial_{x_j}v_0^\gamma\right)\\
&+\varepsilon^2\chi_j^{*,\beta\alpha}(x/\varepsilon,x/\varepsilon^2)\left[\psi_{10\varepsilon}S_\varepsilon\left(\partial_{x_j}v_0^\alpha\right)
-\partial_{y_j}\chi_{k}^{*,\alpha\gamma}(x/\varepsilon)\psi_{10\varepsilon}S_\varepsilon\left(\partial_{x_k}v_0^\gamma\right)\right],
\end{aligned}\end{equation} and
\begin{equation}\begin{aligned}
\tilde{z}_\varepsilon=&\theta_\varepsilon-\theta_0+\varepsilon^2\partial_{x_i}\left(q_{1,ik}^{*,\beta}(y,z)\psi_{10\varepsilon}S_\varepsilon(\partial_{j}u_0^\beta)\right)+\varepsilon \partial_{x_i}\left(q_{2,ik}^\beta(y)\psi_{10\varepsilon}S_\varepsilon(\partial_{j}u_0^\beta)\right)\\
&+\varepsilon^2\partial_{x_i}\left(q_{3,ik}^\beta(y,z)\psi_{10\varepsilon}S_\varepsilon(\partial_{j}u_0^\beta)\right).
\end{aligned}\end{equation}

 Note that $\tilde{w}_\varepsilon^\beta(x)=v_\varepsilon^\beta(x)-v_0^\beta(x)$ and $\tilde{z}_\varepsilon=\theta_\varepsilon-\theta_0$ if $x\in \Omega\setminus\Sigma_{10\varepsilon}$, and Theorem 3.4 yields
\begin{equation}\begin{aligned}||\tilde{w}_\varepsilon||_{H_0^1(\Omega)}+||\tilde{z}_\varepsilon-\int_\Omega \tilde{z}_\varepsilon||_{L^2(\Omega)}\leq & C\varepsilon||\nabla^2 v_0||_{L^2(\Omega)}
+C||\nabla v_0||_{L^2(\Omega\setminus\Sigma_{5\varepsilon})}+C\varepsilon||\nabla v_0||_{L^2(\Omega)}\\
\leq &C\varepsilon^{1/2}||v_0||_{H^2(\Omega)}, \end{aligned}\end{equation} since $\mathcal{L}^*_\varepsilon$ satisfies the same conditions as $\mathcal{L}_\varepsilon$.
 In view of $(3.10)$, we have
\begin{equation}
\begin{aligned} \int_{\Omega} w_{\varepsilon} G d x &=\left\langle\mathcal{L}_{\varepsilon}\left(w_{\varepsilon}\right), v_{\varepsilon}\right\rangle+\int_\Omega w_{\varepsilon} \nabla \theta_{\varepsilon} d x \\ &=\left\langle\operatorname{div}(\tilde{f})+\nabla z_{\varepsilon},
v_{\varepsilon}\right\rangle-\int_{\Omega}\theta_{\varepsilon} \operatorname{div}w_{\varepsilon} d x \\ &=-\int_{\Omega}
\tilde{f} \cdot \nabla \phi_{\varepsilon}-\int_{\Omega}\theta_{\varepsilon} \operatorname{div}\phi dx\\
&=:I_1+I_2, \end{aligned}\end{equation} where  in the last step we use the fact that $\operatorname{div}v_\varepsilon=0$ in $\Omega$. In view of $(3.11)$, there are many terms in $I_1$, but we just give the estimates of some typical terms. Firstly,
\begin{equation}\begin{aligned}
&\int_\Omega\left|\nabla u_0-\psi_{2\varepsilon}S_{\varepsilon}(\nabla u_0)\right||\nabla v_\varepsilon|dx\\
\leq & C \int_\Omega\left|\nabla u_0-S_{\varepsilon}(\nabla u_0)\right||\nabla v_\varepsilon|dx+\int_\Omega\left| S_{\varepsilon}(\nabla u_0) -\psi_{2\varepsilon}S_{\varepsilon}(\nabla u_0)\right||\nabla v_\varepsilon|dx\\
\leq & C\varepsilon||\nabla^2 u_0||_{L^2(\Omega)}||\nabla v_\varepsilon||_{L^2(\Omega)}+C||\nabla u_0||_{L^2(\Omega\setminus \Sigma_{5\varepsilon})}||\nabla v_\varepsilon||_{L^2(\Omega\setminus\Sigma_{4\varepsilon})}.\\
\end{aligned}\end{equation}
According to Lemma 3.3, we have
\begin{equation}
\left|\int_\Omega (H_{21,i}^\alpha+H_{22,i}^\alpha+H_{23,i}^\alpha) \partial_i v_\varepsilon^\alpha dx\right|\leq C\varepsilon||\nabla u_0||_{L^2(\Omega)}||\nabla v_\varepsilon||_{L^2(\Omega)}.\end{equation}
 To estimate the term $\left|\int_\Omega (H_{3,i}^\alpha+H_{4,i}^\alpha) \partial_i v_\varepsilon^\alpha dx\right|$, we just estimate the following typical term:
\begin{equation}\begin{aligned}
&\left|\int_\Omega \varepsilon a_{ih}^{\alpha\beta}\chi_j^{\beta\gamma}(x/\varepsilon)\partial_h\left(\psi_{2\varepsilon}S_\varepsilon
\left(\partial_{x_j}u_0^{\gamma}\right)\right)\partial_i v_\varepsilon dx\right|\\
\leq& C\int_\Omega | \chi(x/\varepsilon)|\cdot|\nabla\psi_{2\varepsilon}|\cdot|S_\varepsilon\left(\nabla u_0\right)|\cdot|\nabla v_\varepsilon |dx+C\varepsilon \int_\Omega | \chi(x/\varepsilon)|\psi_{2\varepsilon}\cdot|S_\varepsilon\left(\nabla^2 u_0\right)|\cdot|\nabla v_\varepsilon |dx\\
\leq& C\varepsilon||\nabla^2 u_0||_{L^2(\Omega)}||\nabla v_\varepsilon||_{L^2(\Omega)}+C||\nabla u_0||_{L^2(\Omega\setminus \Sigma_{5\varepsilon})}||\nabla v_\varepsilon||_{L^2(\Omega\setminus\Sigma_{4\varepsilon})},
\end{aligned}\end{equation} when estimating the other terms in $\left|\int_\Omega (H_{3,i}^\alpha+H_{4,i}^\alpha) \partial_i v_\varepsilon^\alpha dx\right|$, similar estimates will be obtained if we use the Fourier transform methods to separate the different scales of $x$.
In view of $(5.3)$,
\begin{equation}||\nabla v_\varepsilon||_{L^2(\Omega)}\leq\left(||\nabla v_0||_{L^2(\Omega)}+||\nabla \tilde{w}_\varepsilon||_{L^2(\Omega)}+||\nabla(\tilde{w}_\varepsilon- v_0-v_\varepsilon)||_{L^2(\Omega)}\right)\leq C ||v_0||_{H^2(\Omega)},\end{equation}

and according to $(5.5)$, \begin{equation}||\nabla v_\varepsilon||_{L^2(\Omega\setminus\Sigma_{4\varepsilon})}\leq\left(||\nabla \tilde{w}_\varepsilon||_{L^2(\Omega\setminus\Sigma_{4\varepsilon})}+||\nabla v_0||_{L^2(\Omega\setminus\Sigma_{4\varepsilon})}\right)\leq C \varepsilon^{1/2} ||v_0||_{H^2(\Omega)}.\end{equation}
Consequently, according to $(5.7)-(5.11)$, we have
\begin{equation}|I_1|\leq C \varepsilon ||u_0||_{L^2(\Omega)}||v_0||_{H^2(\Omega)}.\end{equation}
Similarly, we can obtain \begin{equation}|I_2|\leq C \varepsilon ||u_0||_{H^2(\Omega)}||\theta_0||_{H^1(\Omega)}.\end{equation} Also, by the $W^{2,2}$ estimates for the Stokes systems with the constant in $C^{1,1}$ domains,
\begin{equation}
||v_0||_{H^2(\Omega)}+||\theta_0||_{H^1(\Omega)}\leq C||G||_{L^2(\Omega)}.
\end{equation} Therefore, combining $(5.12)-(5.14)$ yields that $$||w_\varepsilon||_{L^2(\Omega)}\leq C\varepsilon ||u_0||_{H^2(\Omega)}.$$
In view of the definition $(3.3)$ of $w_\varepsilon$, we can obtain the following error estimates:
$$||u_\varepsilon-u_0||_{L^2(\Omega)}\leq C\varepsilon ||u_0||_{H^2(\Omega)},$$ with the method of Fourier transform, which completes the proof of Theorem 1.1.
\normalem

\end{document}